\documentclass{amsart}
 \usepackage{amsaddr}
 
\usepackage{graphicx, nicefrac}

\usepackage[utf8]{inputenc}
\usepackage[T1]{fontenc}
\usepackage[a4paper,left=2.5cm,right=2.5cm,top=2.5cm,bottom=2.5cm]{geometry}
\usepackage{libertine}
\usepackage{graphicx}
\usepackage{yfonts}
\usepackage{float}
\usepackage{amsmath}
\usepackage{amsfonts}
\usepackage{multicol}
\usepackage{array}
\usepackage{multirow}
\usepackage{listings}
\usepackage{appendix}
\usepackage{subfigure}
\usepackage{mathtools}
\usepackage{bigints}
\usepackage{caption}
\usepackage{amsthm}
\usepackage{bbold}
\usepackage{verbatim}
\usepackage[shortlabels]{enumitem}
\usepackage{tikz-cd}

\makeatletter

\newcommand{\vertiii}[1]{{\left\vert\kern-0.25ex\left\vert\kern-0.25ex\left\vert #1 
    \right\vert\kern-0.25ex\right\vert\kern-0.25ex\right\vert}}
\newcommand{\stars}{}
\DeclareRobustCommand{\stars}[1]{\stars@{#1}}
\newcommand{\stars@}[1]{%
  \ifcase#1\relax\or\stars@one\or\stars@two\or\stars@three\or\stars@four
  \else ??\fi
}
\newcommand{\stars@char}{$\scriptstyle*$}
\newcommand{\stars@base}[1]{%
  $\m@th\vcenter{\offinterlineskip\ialign{\hfil##\hfil\cr#1\crcr}}$%
}
\newcommand{\stars@one}{%
  \stars@base{\stars@char}%
}
\newcommand{\stars@two}{%
  \stars@base{\stars@char\cr\stars@char}%
}
\newcommand{\stars@three}{%
  \stars@base{\stars@char\cr\stars@char\stars@char}%
}
\newcommand{\stars@four}{%
  \stars@base{\stars@char\stars@char\cr\stars@char\stars@char}%
}
\makeatother

\usepackage{amsmath}
\usepackage{cancel}
\usepackage{amssymb}
\usepackage[dvips]{epsfig}
\usepackage{graphicx}
\usepackage{latexsym}
\usepackage{amsthm}
\usepackage{amssymb}
\usepackage{ulem}

\usepackage{amsmath,amsthm,amsfonts,amssymb,amscd,amsbsy,dsfont,hyperref}

\usepackage{palatino}

\usepackage{eucal}
\usepackage{float}
\usepackage{xcolor}

\usepackage{xurl}

\usepackage{color}

\usepackage{multirow}

\newcommand*\dif{\mathop{}\!\mathrm{d}}

\newcommand{\Prob}{\mathbb{P}}

\newcommand{\R}{\mathbb{R}}

\newtheorem*{ack*}{Acknowledgements}
\newtheorem{remark}{Remark}
\newtheorem{corollary}{Corollary}
\newtheorem{definition}{Definition}

\newtheorem{proposition}{Proposition}
\newtheorem{theorem}{Theorem}

\newtheorem{thm}{Theorem}

\newtheorem{Lemma}{Lemma}
\newtheorem{assumption}[thm]{Assumption}

\newtheorem*{assumption*}{Assumption}

\begin{document}

\title{Some limit theorems for locally stationary Hawkes processes \vspace{-1em}}
\author{\footnotesize Thomas Deschatre$^1$, Pierre Gruet$^1$, Antoine Lotz$^{1,2, \dagger}$ $\footnote{\hspace{-0.7em} $\dagger$ \textup{Corresponding author}:  \texttt{antoine.lotz@dauphine.eu}. }$ \vspace{-0.1em}}
\address{\footnotesize	 $^1$\textsc{EDF}-\textsc{Lab} ,  \textsc{FiME}, Palaiseau, France
\\\vspace{-0.2em}
$^2$\textsc{Paris-Dauphine University}, \textsc{PSL}, Paris, France}
\begin{abstract}
     We prove a law of large numbers and functional central limit theorem for a class of multivariate Hawkes processes with time-dependent reproduction rate.  We address the difficulties induced by the use of non-convolutive Volterra processes by recombining classical martingale methods introduced in Bacry \textit{et al.}~\cite{bacrylimit} with novel ideas proposed by Kwan \textit{et al.}~\cite{Kwan1}. The asymptotic theory we obtain yields useful applications in financial statistics. As an illustration, we derive closed-form expressions for price distortions under liquidity constraints.\\
   
\noindent \textbf{Keywords } \vspace{-1em}:  Hawkes processes. Locally stationary processes. Limit Theorems.   
\end{abstract}

\maketitle

\section{Introduction}

\subsection{Motivation and relation to other works}
Hawkes processes are a class of self-interacting point processes. Their applications span numerous fields ranging from neuroscience~\cite{Reynaud}~\cite{Martinez} to financial econometrics~\cite{bacrymicrostructure}~\cite{DeschatreGruet}. In all generality, a multivariate point process $(\boldsymbol{N}_t)=(N_{1,t}, \hdots, N_{p,t})$ records discrete event-times $(t^k_i)$ occurring in continuous time along $p$ components according to
\begin{equation}\label{equ:a_point_process}
    N_{k,t}
    = 
    \sum_{i=1}^{\infty}\mathbb{1}_{ \{ t^k_i \leq t \}}.
\end{equation} 
The law of such process is characterised by its predictable intensity $\lambda_t=(\lambda_{1,t}, \cdots, \lambda_{p,t})$, informally understood as an instantaneous probability of event-time arrival  
\begin{equation}\label{equ:informal_intensity}
    \Prob \big[ N_{k,t} \textup{  jumps in } [t,t + \dif t] \big\lvert \mathcal{F}_t \big]
    =
    \lambda_{k,t} \dif t,
\end{equation}
where $(\mathcal{F}_t)$ is a history of the process. The intensity of the Hawkes process typically takes the form
\begin{equation*}
    \lambda_{k,t}
    =
    \Phi_k \Big[ 
    \mu_k
    +
    \sum_{l=1}^p
    \int_0^t
    \varphi_{kl}(t-s) \dif N_{l,s}
    \Big],
\end{equation*}

where the $\Phi_k \colon \R \mapsto \R^+$ are \textit{activation functions}, the $\mu_k \in \R$ \textit{baseline intensities} and the function $\varphi=(\varphi_{kl}) \colon [0;\infty) \mapsto \R^{p \times p}$ the \textit{kernel} of the process. We refer to section~\ref{section:assumptions} for a rigorous definition. When $p=1$ and $\Phi_1(x)=x$, the dynamic of event-times arrival is that of a population process with Poisson immigration at intensity $\mu_1$ and Galton-Watson reproduction at rate $\lVert \varphi \rVert_{L^1(\dif s)} = \int_0^{\infty} \varphi(s) \dif s$ (see Hawkes \& Oakes~\cite{Oakes}). Under the condition $\lVert \varphi \rVert_{L^1(\dif s)} < 1$, the process is known to reach stationarity as $T \to \infty$, yielding useful approximations and convergence guarantees in statistical inference procedure (see e.g Ogata~\cite{OgataMLE}). It proves limiting in practice however, as the distribution of the phenomenon at hand often fluctuates in time. \\

Generalisations of the Hawkes process have thus arisen in an effort to accommodate time- dependent dynamics. The process is then most often coerced to live on a compact timeframe $[0,T]$ and depends on the double index $(t,T)$.  Chen \& Hall~\cite{FengInferenceForNonStationarySEPP} first considered an univariate linear Hawkes process with varying baseline $ t \in [0,T] \mapsto \mu(\frac{t}{T}) \in [0,\infty)$, allowing the rate of cluster formation to oscillate in time. Roueff \textit{et al.}~\cite{ROUEFF1}~\cite{roueff2} further generalised the model, allowing $\varphi$ to depend on $\frac{t}{T}$ as well -- resulting in the so-called \textit{locally stationary} Hawkes process. This paper explores asymptotic properties of a class of locally stationary Hawkes processes with intensities
\begin{equation}\label{equ:model_intensity}
    \lambda^T_{k,t}
    =
    \mu_k \big(\frac{t}{T}\big) 
    +
    \int_0^t
    g\big(\frac{t}{T} \big)
    \varphi_{kl}(t-s) 
    \dif N_{l,s}, \hspace{0.2cm} t \in [0,T],
\end{equation}

where, for any $k,l= 1\hdots p$, $\mu_k \colon [0,1] \mapsto [0,\infty)$, $g \colon [0,1] \mapsto [0,\infty)$ and $\varphi_{kl} \colon [0,\infty) \mapsto [0,\infty)$. The asymptotic behaviour of the process is well-known in the standard case where $g$ is constant. Functional limit theorems were first proved by Bacry \textit{et al.}~\cite{bacrylimit} for the multivariate linear Hawkes process with constant baseline $(\mu_k)$, and later extended by Deschatre \& Gruet~\cite{DeschatreGruet} to varying baselines.  Their results take the form of a uniform convergence for the normalised process $T^{-1} (
 N_{Tu})_{u \in [0,1]}$ and of the weak convergence of the rescaled process $T^{-\frac{1}{2}} ( N_{Tu}- \mathbb{E} [ N_{Tu}] )_{u \in [0,1]}$ towards a diffusive limit. Analogous results were proven under the mild finite first moment assumption $\int_0^{\infty} t \lVert \varphi \rVert (t) \dif t < \infty $ for non linear activation functions and a constant baseline, by Zhu~\cite{Zhu_nonlinearTCL} in the univariate strictly self-exciting case $\varphi>0$ and by Cattiaux \textit{et al.}~\cite{CATTIAUX2022404} in the more general multivariate and possibly inhibiting case. Strong mixing rates covering the varying baseline case were moreover obtained by Cheysson \textit{et al.}~\cite{Cheysson} under a similar assumption. The asymptotic theory of the varying baseline case is completed by the three recent works of Kwan \textit{et al.}~\cite{Kwan1}~\cite{Kwan2}~\cite{Kwan3}, whom derive general ergodic theorems for statistical inference motives, and introduce a general approximation strategy for non stationary point processes. \\

 Regarding varying kernels finally, a central limit Theorem was obtained by Fierro \textit{et al.}~\cite{Fierro_Leiva_Møller_2015}  in the context of a cluster-like construction of $(N_t)$, wherein the kernel is replaced at each generation by a new function.  As for locally stationary Hawkes processes, neither functional approximations nor ergodic Theorems are known, owing perhaps to the difficulty associated with treating non-convolutive Volterra processes.

 \subsection{Main contribution} We aim to reconcile the model~\eqref{equ:model_intensity} with the general asymptotic theory for Hawkes processes.
Our main results in section~\ref{section:limit_theorems} are a functional law of large number (Theorem~\ref{Th:LLN}) and a central limit Theorem (Theorem~\ref{theorem:FTCL}). Theorem~\ref{Th:LLN} states the uniform convergence
\begin{equation*}
    \sup_{u \in [0,1]}
    \Big\lVert 
    \frac{1}{T}N_{Tu}
    -
    \int_0^u ( \boldsymbol{Id} - \boldsymbol{K}(x))^{-1} \mu(x) \dif x 
    \Big\rVert 
    \xrightarrow[T \to \infty]{}
    0,
\end{equation*}

where $\boldsymbol{K}(x)= g(x) \int_0^{\infty} \varphi(s) \dif s$, and can be regarded as a generalisation of the law of large numbers of Bacry \textit{et al.}~\cite{bacrylimit} and of the ergodic Theorem of Kwan \textit{et al.}~\cite{Kwan2}. Theorem~\ref{theorem:FTCL} establishes the weak convergence in Skorokhod topology of the rescaled process $ T^{-\frac{1}{2}}
    \{ 
    N_{Tu}
    -
    \mathbb{E}\big[ N_{Tu}]
    \}$ , $u \in [0,1]$,  towards the diffusive limit\begin{equation}\label{equ:FTCL_intro}
    \int_0^u (\boldsymbol{Id} - \boldsymbol{K}(x) )^{-1} \boldsymbol{\Sigma}^{\nicefrac{1}{2}}(x) \dif W_x, 
\end{equation}

 where $(W_x)$ is a standard Brownian motion, and $\boldsymbol{\Sigma}$ maps $x \in [0,1]$ to the diagonal matrix with coefficients $(\boldsymbol{Id} - \boldsymbol{K}(x))^{-1} \mu(x)$. This again extends the limit Theorems of Bacry \textit{et al.}~\cite{bacrylimit}, and further generalises their prior extension by Deschatre \& Gruet~\cite{DeschatreGruet}.  Our framework however differs from these two prior approaches, in that the kernel $\kappa(t,s) = g( \nicefrac{t}{T}) \varphi(t-s)$ escapes the convolution-specific techniques on which they depend.  We rely instead on the same broad principle as Kwan, Chen \& Dunsmuir~\cite{Kwan2}. More specifically, while the technical basis for our proofs differs fundamentally from the coupling-like arguments in~\cite{Kwan2}, we find that their approximation strategy based on a grid refining at rate $o(T)$  can be repurposed to produce asymptotic results for the class of non-convolutive Volterra resolvent associated with our kernel.  We isolate in Proposition~\ref{prop:deterministic_convergence} this deterministic result, from which Theorem~\ref{Th:LLN} deduces. Obtaining Theorem~\ref{theorem:FTCL} is more involved.  The key to the proof consists in remarking the sampling scheme in Theorem 3 of Bacry \textit{et al.}~\cite{bacrylimit} refines at the exact same speed  as the approximation grid of the \textsc{kcd} approach. This analogy proves quite productive, as it enables a careful recombination of the two approaches from which Theorem~\ref{theorem:FTCL} follows.

\section{Locally stationary Hawkes processes}\label{section:assumptions}

\subsection{Definition.} We work on a rich enough probability space $(\Omega, \mathcal{F}, \Prob)$ where a family of integer-valued random measures $ \boldsymbol{N} = ( N_1, \cdots, N_p )$ on the real Borel sets $ \mathcal{B}(\R) $ are defined by
\begin{equation*}
    N_k(C) 
    =
    \sum_{n=1}^{\infty}
    \mathbb{1}_{ \{ t_n^k \in C \}},
\end{equation*}
where $(t^1_n)_{n \in \mathbb{N}^\star}, \hdots, (t^p_n)_{n \in \mathbb{N}^\star}$ are random sequences with values in $[0,\infty]$, such that $0 < t^k_{n} < t^k_{n+1} $ for every $n \in \mathbb{N}^\star$ and $k = 1 \hdots p$.
The multivariate point process $\boldsymbol{N}_t = ( N_{1,t}, \cdots, N_{p,t})$ with counting measures $\boldsymbol{N}$ is then defined by $N_{k,t} =  N_k([0,t])$ consistently with~\eqref{equ:a_point_process}.  We will also write $\dif N_{k,s}$ for the measure $N_k(\dif s)$. The history of $\boldsymbol{N}_t$ is the filtration $(\mathcal{F}_t)$ generated by the process until $t$, that is $\mathcal{F}_t = \sigma \{ N(C)  \hspace{0.1cm} 
 \lvert \hspace{0.1cm} C \subset [0,t] \cap \mathcal{B}(\R) \}$. For any $(\mathcal{F}_t)$-progressively-measurable process $\lambda_t= (\lambda_{1,t}, \cdots, \lambda_{p,t})$, $\boldsymbol{N}_t$ is said to admit $(\lambda_t)$ as its intensity if, for any $t>s>0$, $k =  1\hdots p$,
 \begin{equation}\label{equ:martingale}
     \mathbb{E} \big[
        N_k ((s,t]) 
     \big\lvert 
     \mathcal{F}_t
     \big]
     =
     \mathbb{E} \Big[ \int_s^{t}  \lambda_{k,u} \dif u
     \Big\lvert 
     \mathcal{F}_t
     \Big].
 \end{equation}
  Up to some extension of the underlying probability space, we may suppose $p$ independent Poisson measures $\pi_1, \cdots \pi_p$ on $\R^2$ with unit intensity are defined on $(\Omega, \mathbb{F}, \Prob)$, and denote by $(\mathcal{F}_t^\pi)$ their natural filtration. We then have the following existence result (see Brémaud \& Massoulié~\cite{bremaud1996stability}).
\begin{proposition}\label{Proposition:thinning}
      For any $(\mathcal{F}_t^\pi)$--progressively measurable process  $(\lambda_ {k,t})$, the process defined by 
    \begin{equation}\label{equ:thinning}
        N_{k,t}
        =
        \int_0^t \int_0^{\infty}
        \mathbb{1}_{ \{ x \leq \lambda_{k,t} \}}
        \pi_{k,t}( \dif x , \dif s)
    \end{equation}
    admits $(\lambda_{k,t})$ as its intensity. 
\end{proposition}

The process $\boldsymbol{N}_t$ is then said to be \textit{embedded} into the Poisson base $\pi$.  The integral in~\eqref{equ:model_intensity} is taken over $[0,t)$ with $t$ specifically excluded from its domain. The system
\begin{equation*}
    \begin{cases}
        N^T_{k,t}
        &=
        \int_0^t \int_0^{\infty}
        \mathbb{1}_{ \{ x \leq \lambda_{k,s} \}}
        \pi_{k,s}( \dif x , \dif s)
        \\
        \lambda^T_{k,t} 
            &=
              \mu_k\big( \frac{t}{T} \big)
            +
            \sum_{l=1}^p 
            \int_{[0,t)}
                g\big( \frac{t}{T}\big) \varphi_{kl}(t-s)
            N_l(\dif s)
            ,
    \end{cases}
\end{equation*}

recursively defines $(\lambda_{k,t})$ at each $t \in [0,T]$ as a function of the past jumps of the Poisson base $(\pi_k)$. Proposition~\ref{Proposition:thinning} therefore ensures the well-posedness of the following definition.
\begin{definition}\label{def:Hawkes_definition}
    A locally stationary Hawkes process is any point process $\boldsymbol{N}_t$ of the form~\eqref{equ:thinning} with intensity
    \begin{equation*}
            \lambda_{k,t} 
            =
            \mu_k\big( \frac{t}{T} \big)
            +
            \sum_{l=1}^p 
            \int_{[0,t)}
                g\big( \frac{t}{T}\big) \varphi_{kl}(t-s)
            N_l(\dif s).
    \end{equation*}
\end{definition}
Definition~\ref{def:Hawkes_definition} transmits certain properties of the Poisson base to the process itself. Two distinct coordinates of $\boldsymbol{N}_t$ are in particular $\Prob$-a.s prevented from jumping simultaneously. Definition~\eqref{equ:martingale} of the intensity also implies the process
 \begin{equation}\label{equ:fundamental_martingale}
     \boldsymbol{M}_t
     =
     (M_{k,t})
     :=
     \Big( 
     N_{k,t}
     -
     \int_0^t
     \lambda_{k,s}
     \dif s
     \Big), \hspace{0.1cm} t \in [0,T],
 \end{equation}
 is a local martingale. In fact, for any predictable process $ ( H_{k,t})$ such that, $\mathbb{E}[ \int_0^t \lvert H_{k,s} \rvert N_k(\dif s)] < \infty$ for any $k = 1 \hdots p$, one has
 \begin{equation*}
     \mathbb{E}\Big[ 
        \int_0^t 
            H_{k,s} 
        N_l(\dif s)
     \Big\lvert
        \mathcal{F}_t
     \Big]
     =
     \mathbb{E}\Big[ 
        \int_0^t 
            H_{k,s} 
            \lambda_{k,s}
        \dif s
     \Big\lvert
        \mathcal{F}_t
     \Big],
 \end{equation*}
 and we refer the reader to Daley \& Vere-Jones~\cite{DVJ} and Brémaud~\cite{Bremaudbook} for a comprehensive introduction to point processes with stochastic intensities. Taking $H_{k,\cdot}$ to be the indicator function $\mathbb{1}_{[t,t+\dif t ] }$ of some interval $[t,t+\dif t]$, this implies~\eqref{equ:informal_intensity}.

  \subsection{Stability of the process.} For any $d \in \mathbb{N}^{\star}$ and any $x \in \R^d$, we denote by $\lVert \cdot \rVert_q$ the $\ell^q$-norm  $\lVert x \rVert_q=( \sum_i^d \lvert x_i \rvert^q)^{\nicefrac{1}{q}}$. For any $f \colon (a,b) \mapsto \R^d$, where $a, b \in \R \cup \{- \infty,\infty \}$, we write $\lVert f \rVert_{L^{\infty}[a,b]}$ for  $\sup_{x \in (a,b)} \lVert f(x)  \rVert_1$, and omit $(a,b)$ from the notation where there is no ambiguity. Likewise, we use the notation $\lVert f \rVert_{L^q} = (\int_{(a,b)} \lVert f(x) \rVert_q^q \dif x)^{\nicefrac{1}{q}}$ for the $L^q(a,b)$-norm. Finally, $\mathcal{M}_d(\R)$ is the set of $d \times d$ matrices with real coefficients.  
  
  \begin{assumption}\label{ass:g_and_mu_are_C0}
      The baseline and reproduction functions  $\mu \colon  [0,1] \mapsto [0,\infty)^p$, $g \colon [0,1] \mapsto [0,\infty)$ are respectively continuous and continuously differentiable over $[0,1]$.
  \end{assumption}

  \begin{assumption}[Stability]\label{ass:stability}
      For any $k,l=  1 \hdots p$, $\int_0^{\infty}  \varphi_{kl}(t) \dif t < \infty$, and the matrix with coefficients  
      \begin{equation*}
           \lVert g \rVert_{L^{\infty}} \Big( 
           \int_0^{\infty}  
            \varphi_{kl}(t)  
        \dif t \Big) 
      \end{equation*}
      
      has spectral radius below $1$.
  \end{assumption}

    \begin{remark}
        When $\mu$ and $g$ are constants, Assumption~\ref{ass:stability} reduces to the classical stability condition in Bacry  \textit{et al.}~\cite{bacrylimit} or Brémaud \& Massoulié~\cite{bremaud1996stability}.  Note also that while we have not used the cluster representation of Roueff \textit{et al.}~\cite{ROUEFF1}, Assumption~\ref{ass:stability} is comparable to~\cite[Theorem 1]{ROUEFF1}.
  \end{remark}
  
  For any $k,l=1 \hdots p$, define the integrable function
  \begin{equation}\label{equ:bound_function}
      \Bar{\varphi}_{kl} \colon t \in [0,\infty)
      \mapsto 
     \lVert g \rVert_{L^{\infty}}  \varphi_{kl}(t),
  \end{equation}

  and, accordingly, $\bar{\varphi}=(\bar{\varphi}_{kl}) \colon [0,\infty) \mapsto  \mathcal{M}_p(\R)$. Assumption~\ref{ass:stability} guarantees the stability of the process in the sense that $N^T_k(A)< \infty$ for any bounded $A \in \mathcal{B}(\R)$ and any $k = 1 \hdots p$. The random measures $N^T_k$ are then said to be locally finite.  This is clear from the fact the jumps $(t^k_i)$ of $(\boldsymbol{N}^T_t)$ are comprised in the jumps of a standard Hawkes process $(\bar{\boldsymbol{N}}_t)$ with baseline $(\sup_k \mu_k(x))$ and kernel $\bar{\varphi}$  embedded in the same Poisson base $(\pi_k)$ . For any $k=1 \hdots p$, the process $(N^T_{k,t})$ is then upper-bounded by $(\bar{N}_{k,t})$, which admits a finite first moment at any $t>0$, see for instance~\cite[Lemma 2]{bacrylimit} or Delattre \textit{et al.}~\cite[Lemmata 23 \& 24]{DelattreHawkesLargeNetworks}. In particular, one has Lemma~\ref{lemma:process_if_well_defined}.
  \begin{Lemma}\label{lemma:process_if_well_defined}
      Suppose Assumptions~\ref{ass:g_and_mu_are_C0} and~\ref{ass:stability} hold. Then, the counting measures $N^T_k$, $k=1\hdots p$ are all locally finite, and, furthermore, $\mathbb{E}[N^T_{k,t}] < \infty$ for any $T>0$ and any $t \in [0,T]$.
  \end{Lemma}

  With all existence guarantees provided we proceed to our main results.

\section{Law of large number and central limit theorem.}\label{section:limit_theorems}

Let $(N^T_t)_{t \in [0,T]}$ be a collection indexed on $T$ of locally stationary Hawkes processes with reproduction function $g$, kernel $\varphi$ and baseline $\mu$. Define the matrix function
\begin{equation*}
    \boldsymbol{K} \colon x \in [0,1] \mapsto
    g(x)\Big(  \int_0^{\infty} \varphi(s) \dif s \Big) \in \mathcal{M}_p(\R^+).
\end{equation*}

When $g$ and $\mu$ satisfy Assumption~\ref{ass:stability}, the matrix $\boldsymbol{Id}-\boldsymbol{K}(x)$ is invertible for any $x \in [0,1]$, and one has following functional law of large numbers.

\begin{theorem}[Law of Large Numbers]\label{Th:LLN}
    Suppose Assumptions~\ref{ass:g_and_mu_are_C0} to~\ref{ass:stability} hold. Then,
    \begin{equation*}
    \sup_{u \in [0,1]}
        \Big\lVert
            \frac{1}{T} N_{Tu}^T
            -
            \int_0^u
                ( \boldsymbol{Id} - \boldsymbol{K}(x))^{-1} \mu(x)
            \dif x
        \Big\rVert
        \to
        0
        \end{equation*}
        in $L^2(\Prob)$ and $\Prob$-a.s.
\end{theorem}
As we move to the weak convergence of the rescaled process, Assumption~\ref{ass:stability} needs be slightly strengthened into Assumption~\ref{ass:disease}.

\begin{assumption}\label{ass:disease}
    There exists $\epsilon \in (0,\infty]$ such that $\varphi \in L^{1+\epsilon}[0,\infty)$
\end{assumption}
\begin{theorem}[Central limit theorem]\label{theorem:FTCL}
    Under Assumption~\ref{ass:g_and_mu_are_C0} to~\ref{ass:stability},
    \begin{equation*}
        \frac{1}{\sqrt{T}}
        \big\{ 
        N_{Tu}^T
        -
        \mathbb{E}[N_{Tu}^T]
        \big\}
        \xrightarrow[T \to \infty]{\mathcal{L}(\Prob)}
        \int_0^u (\boldsymbol{Id}-\boldsymbol{K}(s) )^{-1}
        \boldsymbol{\Sigma}^{\frac{1}{2}}(s) \dif W_s
    \end{equation*}
   in the sense of finite-dimensional distributions, where $(W_s)_{s \in [0,1]}$ is a standard Brownian motion and
    \begin{equation*}
    \boldsymbol{\Sigma}
    \colon x
    \mapsto 
    \textup{Diag}\big((
         \boldsymbol{Id}-\boldsymbol{K}(x) )^{-1} \mu(x)  \big).
    \end{equation*}
    If furthermore Assumption~\ref{ass:disease} holds, the convergence occurs in Skorokhod topology.
\end{theorem}

Under some mild additional conditions on $m$ and $\varphi$, Theorem~\ref{theorem:FTCL} enjoys an explicit extension.

\begin{assumption}\label{ass:Lipschitz}
    The baseline function $\mu$ is Lipschitz-continuous over $[0,1]$.
\end{assumption}

\begin{assumption}\label{ass:sqrt_integ}
    For any $k,l= 1 \hdots p$, $\int_0^t \varphi_{kl}(s) \sqrt{s}\dif s< \infty$.
\end{assumption}

\begin{corollary}\label{coro:sqrt_coro}
Under Assumption~\ref{ass:g_and_mu_are_C0} to~\ref{ass:sqrt_integ},
     \begin{equation*}
        \frac{1}{\sqrt{T}}
        \big\{ 
        N_{Tu}^T
        -
        \int_0^u
                ( \boldsymbol{Id} - \boldsymbol{K}(x))^{-1} \mu(x)
            \dif x
        \big\}
        \xrightarrow[T \to \infty]{\mathcal{L}(\Prob)}
        \int_0^u (\boldsymbol{Id}-\boldsymbol{K}(s) )^{-1}
        \boldsymbol{\Sigma}^{\frac{1}{2}}(s) \dif W_s
    \end{equation*}
    in Skorokhod topology.
\end{corollary}

\begin{remark}
     Theorem~\ref{theorem:FTCL} is not a \textit{stricto sensu} generalisation of Bacry \textit{et al.}~\cite[Theorem 2]{bacrylimit} and  Deschatre \& Gruet~\cite[Proposition 8]{DeschatreGruet} as they do not rely on Assumption~\ref{ass:disease} to obtain a functional convergence.  In practical terms however, virtually all kernels encountered in the literature verify Assumptions~\ref{ass:stability} and~\ref{ass:disease} alike. It suffices that $\varphi$ be bounded. Hence this includes the exponential $\varphi(t) \propto \exp( - \beta t)$, power law, $\varphi(t) \propto (t+\gamma)^{- \beta}$, gamma $\varphi(t) \propto t^{\gamma} exp( - \beta t)$, or histogram kernel $\varphi(t) = \sum_i \alpha_i \mathbb{1}_{[ib,(i+1)b]}$.
\end{remark}

\begin{remark}\label{remark:multivariate_g}
    There is no obstacle to considering a multivariate $g$ in lieu of a scalar one. With a kernel of the form $(x,s) \mapsto g_{kl}(x) \varphi_{kl}(s)$, Assumption~\ref{ass:stability} becomes $\rho( (\sup_{x} g_{kl}(x) \int_0^{\infty} \varphi_{kl} (s) \dif s)_{kl})<1$ where $\rho$ is the spectral radius, Theorem~\ref{theorem:FTCL} holds with the outer product replaced by the coordinate-wise product, and the proofs remain intact up to some minor technical details -- see section~\ref{section:multivariate_g}. 
\end{remark}

One can remark that while the adjunction of a varying reproduction rate renders most of the process' characteristics intractable, even in the simple exponential case $\varphi(t) \propto \exp(  - \beta t)$, the limits in Theorems~\ref{Th:LLN} and~\ref{theorem:FTCL} remain computable in terms of exact entries of the model.

\section{Application to financial statistics}\label{section:application}

\subsection{A model for prices under liquidity constraints}We revisit the microstructure model of Bacry \textit{et al.}~\cite{bacrymicrostructure}. Suppose one observes the price $(P_t)$ of an asset over $t \in [0,T]$. In the context of high-frequency trading, $(P_t)$ is recorded at the tick-by-tick level and the price trajectory bears a discrete nature. A similar property is retained at lower frequencies by the less liquid assets. A fitting models class then takes the form $P_t = N^+_t -N^-_t$, where $(N^{+}_t)$, respectively $(N^{-}_t)$, is a pure jump process recording positive, respectively negative, price increments. In~\cite{bacrylimit}, $(N^{+}_t,N^{-}_t)$ is a bivariate Hawkes process with baseline intensity and kernel 
\begin{equation}\label{equ:bacrymodel}
\mu= \begin{bmatrix}
    \mu^+\\
    \mu^-
\end{bmatrix}
\textup{  and  }
    \varphi(t) = \begin{bmatrix}
        0 & \alpha \\
        \alpha & 0\\
    \end{bmatrix}
    \Phi(t,\beta),
\end{equation}

where $\Phi(t,\beta)= \beta \exp(-\beta t)$ with $(\beta,\alpha) \in [0,\infty) \times [0,1)$, and $\mu^+=\mu^-$. The antidiagonal structure of $\varphi$ mimics microstucture noise and introduces reversion effects in price dynamics, meaning that an upward price increment is more likely followed by a downward price move and conversely, see the references in~\cite{bacrylimit}. Such behaviour is expected in the presence of a diversified pool of traders implementing broadly uncorrelated strategies. It may also be partially attributed to market makers quoting at both sides of the order book.   Should a large fraction of all participants retreat from the market, one would expect $\alpha$ to decrease as the remaining traders' views may align more often. For low values of $\alpha$, prices tend to move in the same direction one wishes to execute in, resulting in a worse outcome for liquidity takers. In this sense, the reproduction rate $\alpha$ measures the rate at which liquidity replenishes. \\

That participation in the market may vary as time passes is not a purely theoretical scenario. In the context of commodities futures, some agents may be specifically adverse to the physical risk associated with taking delivery, and the typology of traders involved in the market may thus evolve as expiry looms by. What is empirically known and recognised as a stylised fact of commodities markets is the existence of the \textit{Samuelson effect}, which describes a decreasing relation between volatility and time to maturity. We will return to this precise matter further down this section, and refer for now to Jaeck \& Lautier~\cite{JAECK2016300} and Aïd \textit{et al.}~\cite{Aid} for more details.  \\

We consider situations where some price-insensitive directional flow moves $(P_t)$, corresponding to an imbalance $\mu^+ \neq \mu^-$ in $\mu$, while liquidity progressively dries out as $t \to T$ with
\begin{equation}\label{equ:quadractic_decrease}
    \alpha( x)
    =
    \alpha_0 (1- \eta  x^2),
\end{equation}
$0 \leq \eta < 1$. The function $\alpha$ is parameterised so that $\alpha(0) =\alpha_0$ and $\alpha(1)=(1-\eta)\alpha_0<\alpha(0) $.  Say, without loss of generality, that $\mu^+= \mu_0 + \Delta \mu > \mu_0= \mu^-$. The modified price model is then simply the one of~\cite{bacrymicrostructure}, with the adjunction of additional Poisson jumps $(N^{\Delta \mu}_t)$ with intensity $\Delta \mu$ to $(N^+_t)$, where $(N^{\Delta \mu}_t)$ is independent from $(N^+_t,N^-_t)$.  The market reacts to the increased buying activity according to its usual dynamic. The resulting process $(N^{+}_s(\eta,\Delta \mu),N^{-}_ \mu(\eta,\Delta \mu))$ has intensity
\begin{align*}
    \lambda^{+}_t(\eta,\Delta \mu) &=  \mu_0 + \Delta \mu + \int_0^t \alpha(\frac{t}{T}) \Phi(t-s,\beta) \dif N^{-}_s(\eta,\Delta \mu) \\
    \lambda^{-}_t(\eta,\Delta \mu) &=  \mu_0  + \int_0^t \alpha(\frac{t}{T}) \Phi(t-s,\beta) (\dif  N^{+}_s + \dif N^{\Delta \mu}_s),
\end{align*}

where $N^+_t(\eta,\Delta \mu)= N^{+}_t +  N^{\Delta \mu}_t$. It is again a (locally stationary) Hawkes process.  With $\Delta \mu$ set and given, we denote by
\begin{equation*}
    P_t^\eta 
    =
    N^+_t(\Delta \mu,\eta) 
    -
    N^- (\Delta \mu,\eta) 
\end{equation*}

the modified price. The additional flow may be roughly conceptualized as a single trader executing a large buy meta-order of size $Q$ via a \textsc{twap}\footnote{Time weighted average price, see Almgren \& Chriss~\cite{AG_01} for an equivalent class of strategies.} algorithm, which should induce a move of size $\Delta \mu \propto \sqrt{Q}$. We wish to characterize the market's capacity to absorb the increased demand as a function of $\eta$. Consistently with this interpretation, we consider the additional slippage due to $\alpha$ decreasing as time passes (see Remark~\ref{remark:slippage}). That is, the proportion of marginal price distortion which can be attributed to variations in $\alpha$ for each additional unit of demand $\Delta \mu$ .
\begin{proposition}\label{lemma:price_model}
Under Assumptions~\ref{ass:g_and_mu_are_C0} to~\ref{ass:disease}, for any $u \in [0,1]$ and any $\eta \in [0,1]$,
    \begin{equation*}
        \frac{\mathbb{E}\big[ P^{\eta}_{Tu}-P^\eta_{0}\big]}
        {\mathbb{E}\big[ P^{0}_{Tu} - P^0_0\big]}
        \to 
        \frac{
        \textup{atanh} ( \sqrt{\zeta_0} u) }{\sqrt{\zeta_0} u}
        \textup{  as  }
        T \to \infty
    \end{equation*}
    where $\zeta_0=\alpha_0 \eta (1+\alpha_0)^{-1}$.
\end{proposition}

\begin{proof}
    For any $\alpha \colon [0,1] \mapsto [0,1)$, and any baseline intensity of the form $\mu^+ - \mu^- = \Delta \mu$, by Theorem~\ref{Th:LLN}, the expected return $T^{-1}\mathbb{E}\big[ P_{Tu}-P_0\big]$ of a locally stationary Hawkes price model $(P_t)$ with antidiagonal structure~\ref{equ:bacrymodel} and reproduction function $\alpha$ converges towards
    \begin{align*}
        \int_0^u 
        \frac{(\mu^+ +  \alpha(s) \mu^-)-
            (\mu^- +  \alpha(s) \mu^+)}{1-\alpha(s)^2}  
        =
        \int_0^u
        \frac{\Delta \mu (1-\alpha(s))}{(1-\alpha(s))(1+\alpha(s))}
        \dif s
        =
        \Delta \mu
        \int_0^u
        \frac{1 }{(1+\alpha(s))}
        \dif s.
    \end{align*}
    The relative slippage is the ratio of $T^{-1}\mathbb{E}\big[ P^{\eta}_{Tu}-P^{\eta}_0\big]$ and $T^{-1}\mathbb{E}\big[ P^{0}_{Tu}-P^{0}_0\big]$ hence it converges to
    \begin{equation*}
       \Big( \int_0^u \frac{\Delta \mu}{(1+\alpha_0)(1-\zeta_0 s^2)} \dif s \Big)  \Big( \frac{ u\Delta \mu }{1+\alpha_0}\Big)^{-1}
    \end{equation*}
    when $\alpha$ takes the shape~\eqref{equ:quadractic_decrease}
    which is the desired result.
\end{proof}

\begin{remark}
    The proof scheme for Proposition~\ref{lemma:price_model} remains valid for a varying baseline $\mu \colon [0,1] \mapsto [0,\infty)^2$.  In particular, if $\mu^+(t) - \mu^-(t) > \Delta \mu$ for any $t \in [0,1]$, then, an asymptotic lower bound for the relative slippage is still $\textup{atanh}( \sqrt{\zeta_0}u) (\sqrt{\zeta_0}u)^{-1}$.
\end{remark}

\begin{remark}\label{remark:slippage}
    The term "slippage" is borrowed from the financial literature, where it refers to the 
 difference between the price at which one intends to execute a marginal order, and the price at which it is actually filled. Simple models express slippage as a function of volatility, and the present one does in fact establish such a link. However, $\alpha$ being a rate of change rather than an absolute measure of liquidity, the relation is not quite straightforward. From the same arguments as Bacry \textit{et al.}~\cite{bacrylimit}, the limit volatility $\sigma$ of $(P_t)$ expresses as 
    \begin{equation*}
        \sigma(u)= \frac{\mu^++\mu^-}{(1-\alpha(u))(1+\alpha(u))^2},
    \end{equation*}

    a decreasing function of $\alpha$ for $\alpha \leq \frac{1}{3}$ and an increasing function for $\alpha \geq \frac{1}{3}$.  The somewhat indirect link with volatility suggests some form of trade-off. When $\alpha$ decreases, directional price movement become more persistent, and $\sigma$ decreases too. If $\alpha$ decreases further however, the amplitude of directional variations increases beyond a certain threshold, and the relation reverts.
\end{remark}

A consequence of remark~\ref{remark:slippage} is that variations in $\alpha$ primarily affect the persistence of trends in the dynamics of $(P_t)$, whilst their relation with volatility is more ambiguous. Hence Proposition~\ref{lemma:price_model} should \textit{not} be understood as a fundamental basis for the Samuelson effect. An increase in volatility as $t \to T$ is better represented by an increase in the amplitude of $\mu(x)$ as $x \to 1$, which bears an unequivocal amplifying effect on $\sigma$, and which our results do allow. We refer to Deschatre \& Gruet~\cite{DeschatreGruet} for such a model, where $\mu(x) \propto \exp( \kappa x)$ for some $\kappa>0$. In conclusion, a price model of the type~\eqref{equ:bacrymodel} will only clearly attribute the Samuelson effect to exogenous price moves, and we expect that this characteristic reflects actual market dynamics to some degree.  While we are not immediately able to confront our results with market data, we may still evaluate our approach via numerical simulations, to which we now proceed. 

\subsection{Numerical experiments} Consider first an univariate locally stationary Hawkes process with reproduction function $\alpha$, unit baseline $\mu=1$, and exponential kernel $\varphi(t) = \exp(-t)$. We perform i.i.d simulations $N^1_t,  \hdots, N^n_t$ of the process for the two choices $\alpha(x)= 0.8\exp( - 10 (x-\frac{1}{2})^2)$ and $\alpha(x)=0.8x$. In figure~\ref{fig:MC1} we compare the Monte-Carlo estimator $\frac{1}{n}\sum_{k=1}^n N^k_t$ with the theoretical limit of Theorem~\ref{Th:LLN}.

\begin{figure}[H]
    \centering
    \includegraphics[width=0.35\linewidth]{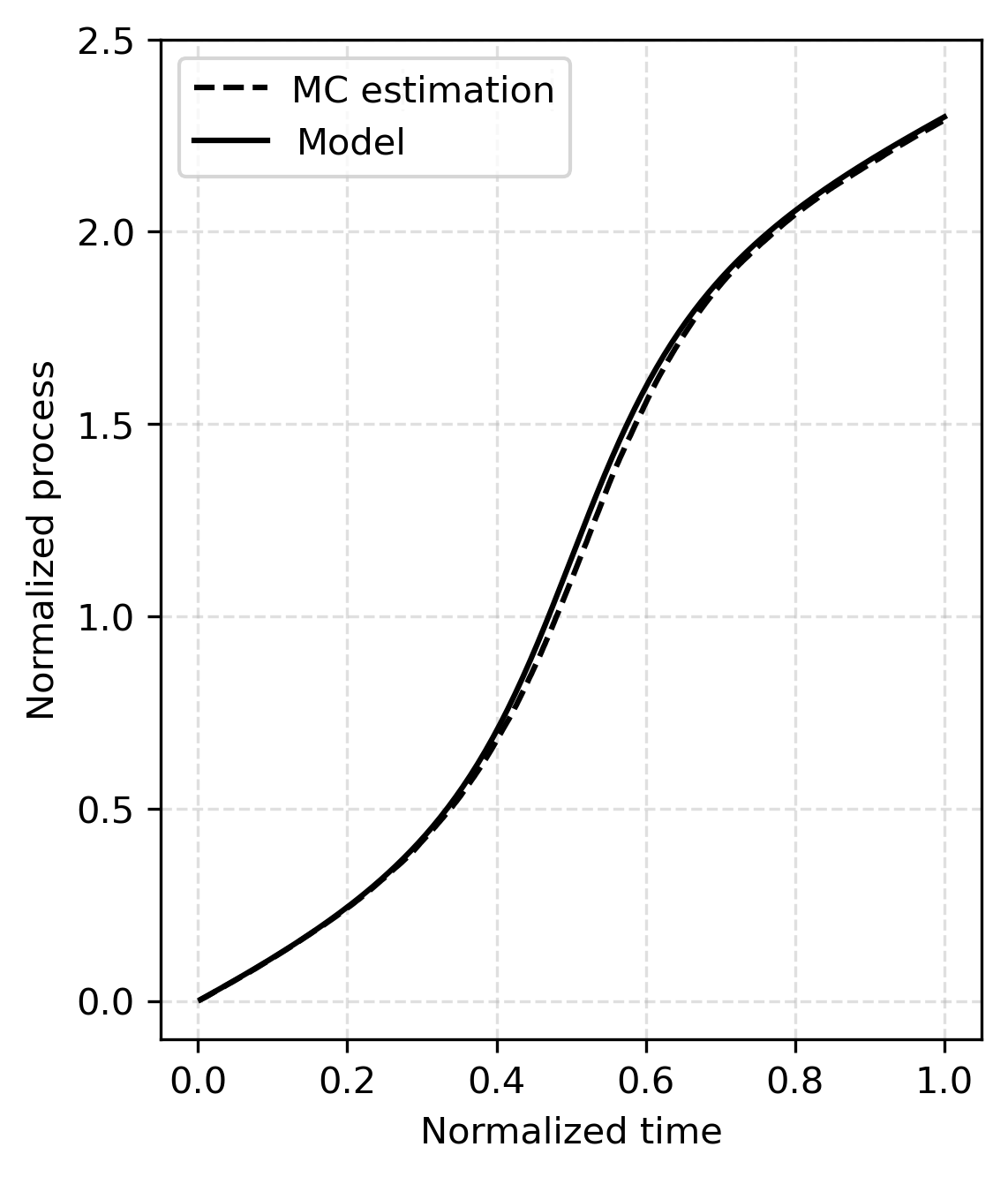}
    \includegraphics[width=0.35\linewidth]{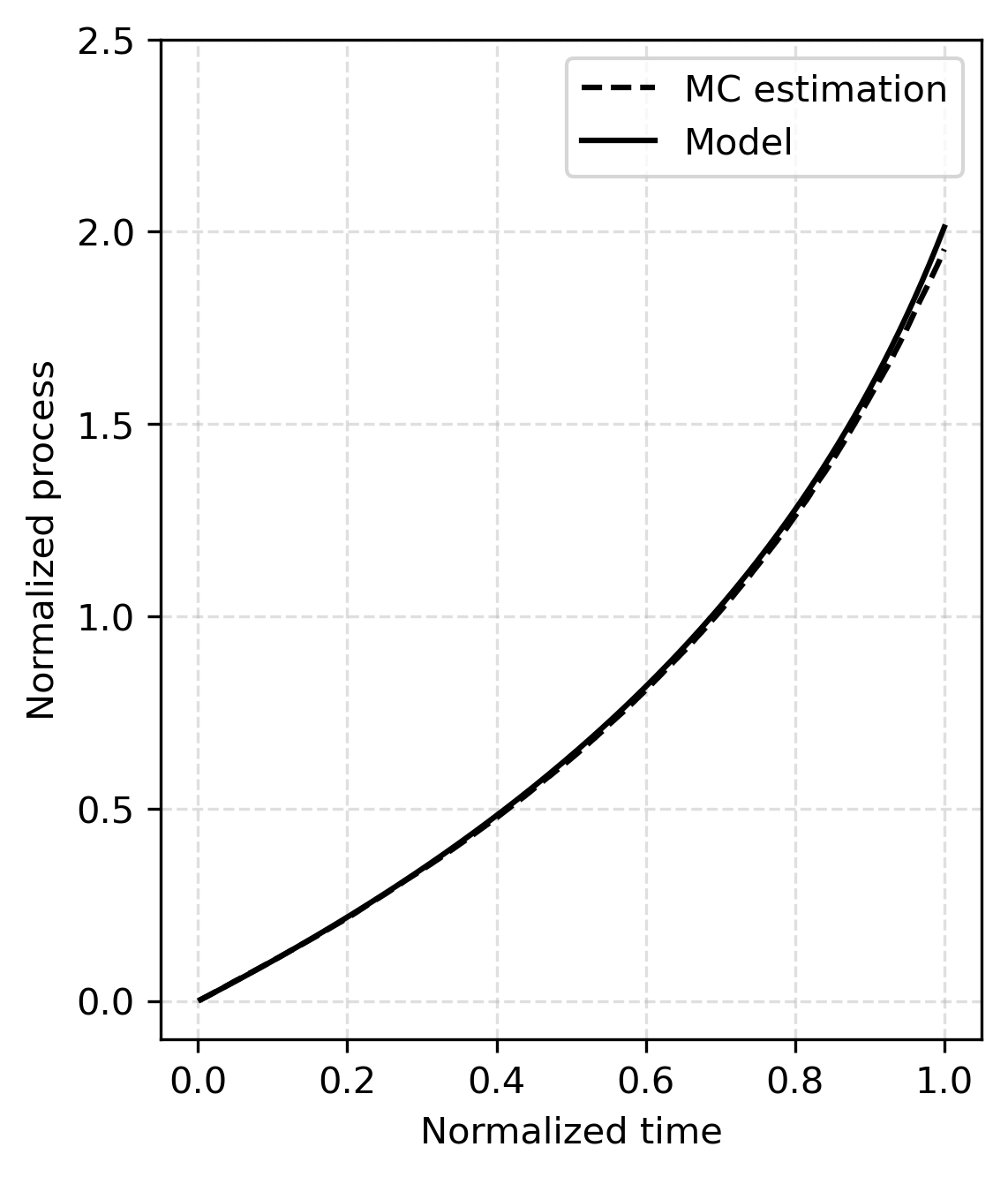}
    \caption{Theoretical limit of the normalized process $T^{-1}N^T_t$ as a function of the normalized time $\frac{t}{T}$ (solid line), as compared to its empirical estimation  $n^{-1} \sum_{k=1}^n T^{-1} N^{k,T}_t$ on $n=500$ independent simulations $(N^{k,T}_t)$ of the process with $T=200$ (dotted line), with a Gaussian reproduction function (left) and a linear reproduction function (right).}
    \label{fig:MC1}
\end{figure}

The figure confirms the good fit of Theorem~\ref{Th:LLN} with the empirical behaviour of the process. We proceed to our price model, and confront Lemma~\ref{remark:slippage} with synthetic data. We simulate $n=3000$ trajectories of a price model~\eqref{lemma:price_model} with the varying reproduction rate~\eqref{equ:quadractic_decrease} for three different values of $\eta$, and compare the average price with its theoretical counterpart in figure~\ref{fig:MC2}. On its sides are also represented the theoretical slippage and reproduction rate.

\begin{figure}[H]
    \centering
    \includegraphics[width=0.3\linewidth]{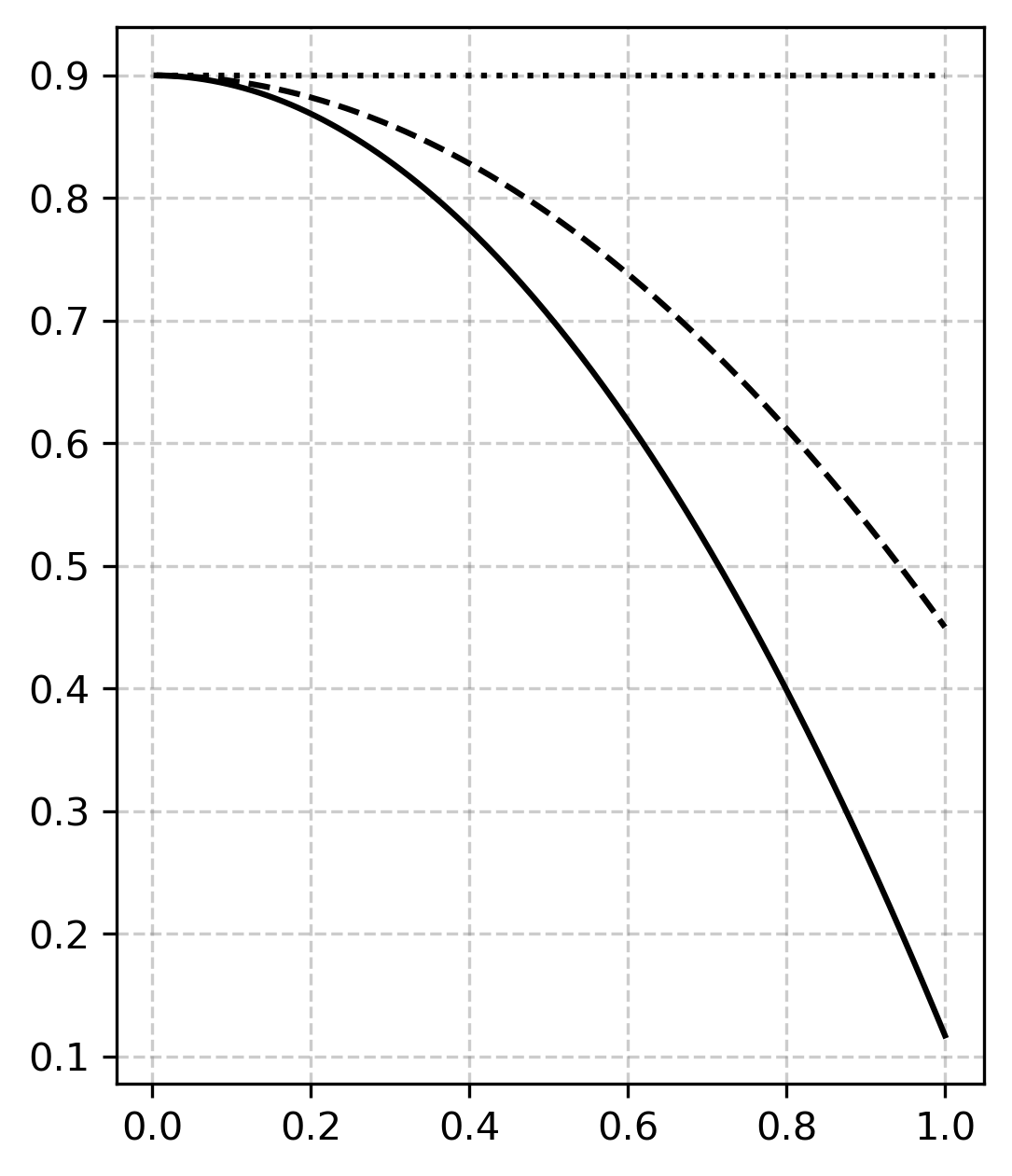}
    \includegraphics[width=0.3\linewidth]{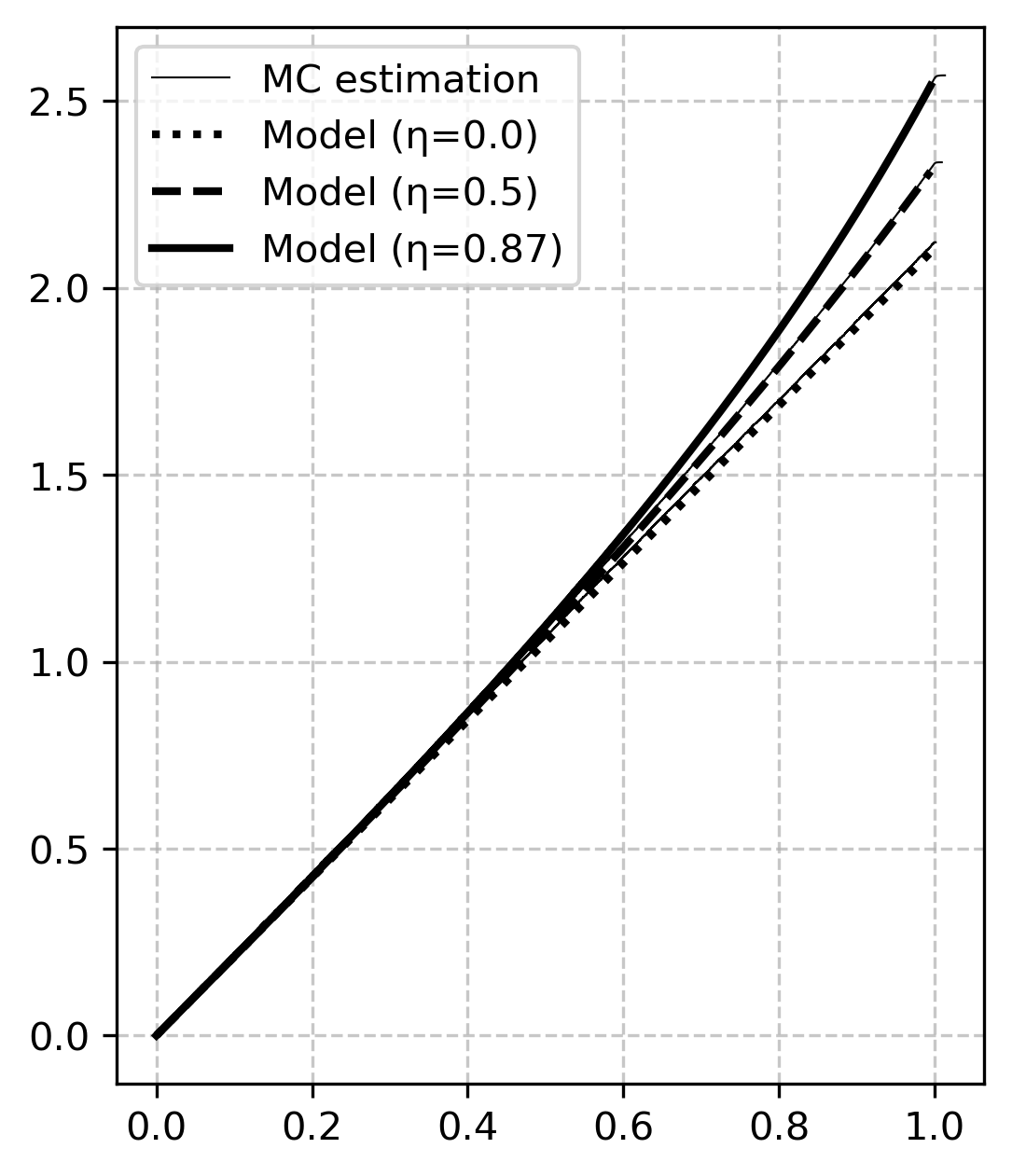}
    \includegraphics[width=0.3\linewidth]{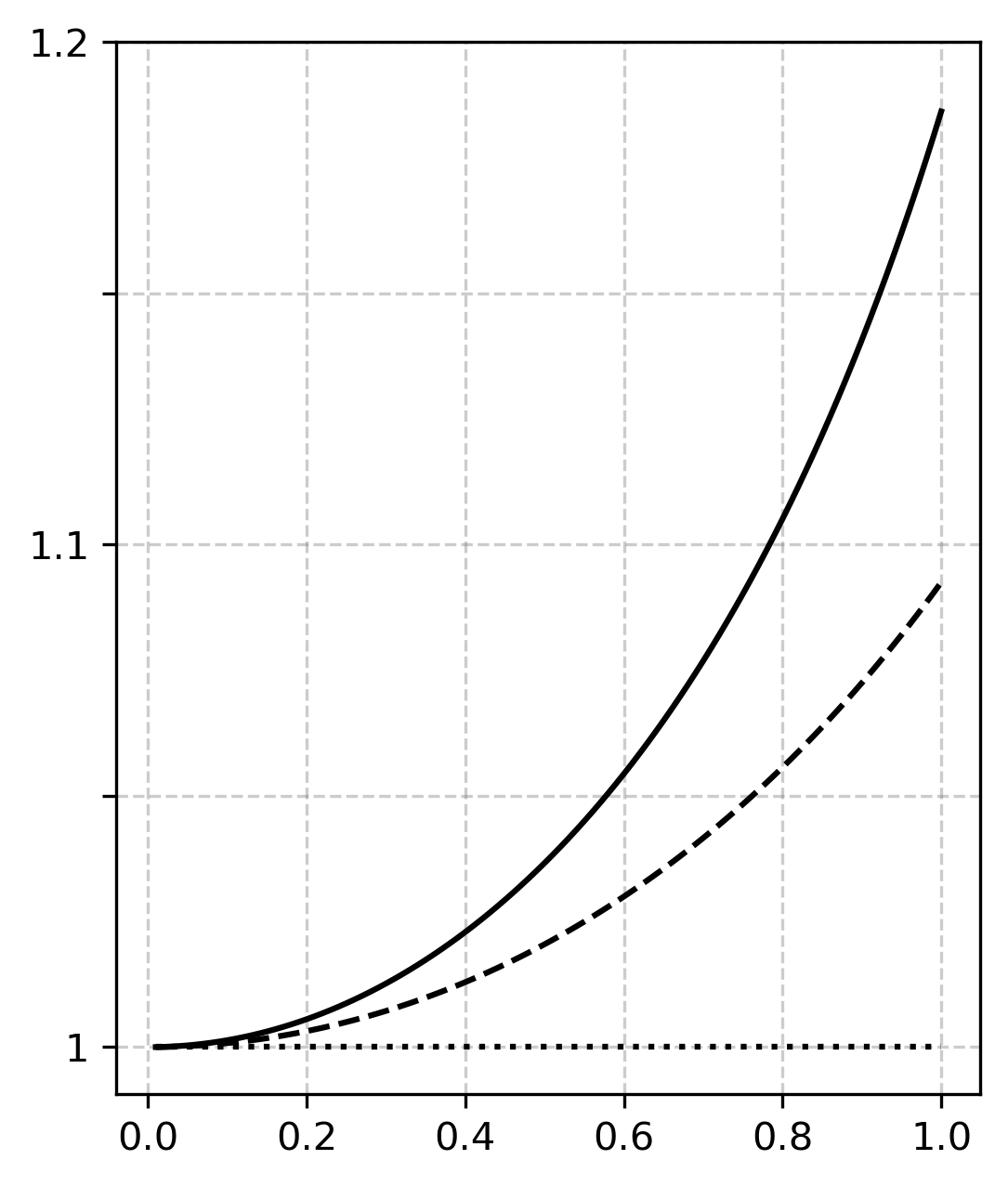}
    \caption{Reproduction function $\alpha$ (left), Expected return (middle), and Slippage (Right) for a locally stationary Hawkes prices model. Expected returns are computed from Theorem~\ref{Th:LLN} and compared to the Monte-Carlo estimator $n^{-1} \sum_{k=1}^n T^{-1} P^{k,T}_t$ for $n=3000$ independent simulations. }
    \label{fig:MC2}
\end{figure}

\section{Discussion}

We have established in section~\ref{section:limit_theorems} novel functional law of large number and central limit Theorem for locally stationary Hawkes processes, and presented in Section~\ref{section:application} some of their applications to statistical finance. The theoretical side of our work opens some interesting perspectives. Theorem~\ref{theorem:FTCL} provides a proof of the existence of a globally defined weak limit for a class of Hawkes processes admitting no stationary distribution, such feature being of possible interest in the context of statistical inference. It is also possible to envision a \textit{marked} extension of Theorem~\ref{theorem:FTCL}, within which external covariates are allowed to influence the law of the process. As a matter of fact, the introduction of a time-dependent reproduction rate may be regarded as the adjunction of non-stationary deterministic marks to the kernel of the process.\\

Regarding the applied part of the article, we acknowledge that Proposition~\ref{lemma:price_model} is dependent upon our supposition of a quadratic decay in the amplitude of the reproduction rate. Some further statistical work and estimation on market data are thus required to assess the credibility of such hypothesis. We refer to Roueff \textit{et al.}~\cite{roueff2} and Mammen \& Müller~\cite{Mammen} for existing developments on the matter.

\section{Preparation for the proofs}

\subsection{On Volterra equations}

We first recall, mainly from Gripenberg~\cite[Chapter 9]{gripenberg}, some classical results and notations for Volterra integral equations in general, and renewal equations in particular. A Volterra kernel is any function $k \colon (s,t)  \in \R^2 \mapsto \mathcal{M}_p(\R)$ such that $k(t,s)=0$ when $s>t$. If $a$ and $b$ are two Volterra kernels and $f$ some Borel function over $[0,T]$, the $\star$ products over Volterra kernels are defined by
\begin{align*}
    (a \star b)(t,s) = \int_s^t a(t,u) b(u,s) \dif u
    \hspace{0.2cm} \textup{and} \hspace{0.2cm}
    (a \star f) (t) = \int_0^t a(t,u) f(u) \dif u.
\end{align*}

We will be concerned with functional equations of the form
\begin{equation*}\label{equ:general_Volterra_equation}
    x(t) = m(t) + k \star x (t), \hspace{0.1cm} t \in [0,T],
\end{equation*}

where $k \colon [0,T]^2 \mapsto \mathcal{M}_p(\R)$ is a Volterra kernel and $m \colon [0,T] \mapsto \R^p$ a bounded function. When it is defined, the resolvent of a Volterra kernel $k$ is the series $K=\sum_{i=1}^{\infty} k^{\star i}$ where $k^{\star i}$ is the $i$-fold  $\star$ product of $k$ with itself, with the convention $k^{\star 1}= k$. In the sequel, we will work with the kernel
\begin{equation*}
    k^T(t,s) = g\big( \frac{t}{T} \big) \varphi(t-s),
\end{equation*}
where we have identified $\varphi$ to a function of $\R$ with support in $[0,\infty)$. The kernel $k^T$ verifies for any $ (s,t) \in [0,T]^2$ the coordinate-wise inequality $k^T(t,s) \leq \Bar{\varphi}(t-s)$, where the bounding kernel $\Bar{\varphi}$ is defined in~\eqref{equ:bound_function}. Thus $\sum_{i=1}^n (k^T)^{\star i}(t,s) \leq \sum_{i=1}^n \bar{\varphi}^{\star i}(t-s)$ for any $n \in \mathbb{N}^{\star}$.  As $\int_0^{\infty} \bar{\varphi}^{\star i}(t) \dif t = ( \int_0^{\infty} \bar{\varphi}(t) \dif t)^i$, the series of the $\bar{\varphi}^{\star i}$  converges in $L^1[0,\infty)$. Hence the series of the $(k^T)^{\star i}(t,s)$ converges for almost every $(s,t) \in [0,T]^2$, and the kernel $k^T$ admits a well defined resolvent $K^T = \sum_{i=1}^{\infty} (k^T)^{\star i}$.

\begin{Lemma}[Theorem 3.6 in Chapter 9 of Gripenberg~\cite{gripenberg}]\label{Lemma:gripenberg}
    Let $T>0$, and $m \colon [0,T] \mapsto \R^p$. Then any function $x \colon [0,T] \mapsto \R$ such that \begin{equation}\label{equ:our_volterra_equation}
        x(t) 
        \leq 
        m(t)
        +
        \int_0^t 
        k^T(t,s) x(s) \dif s, \hspace{0.1cm} t \in [0,T]
    \end{equation}
    where $k^T(t,s) 
        =
        g(\frac{t}{T}) \varphi(t-s)$
    also verifies \begin{equation}\label{equ:our_volterra_solution}
        x(t) 
        \leq 
        m(t)
        +
        \int_0^t
        K^T(t,s) m(s) \dif s,
    \end{equation}
    with 
        $ K^T(t,s) =
        \sum_{i=1} (k^T)^{\star k}(t,s) $. If furthermore equality holds in~\eqref{equ:our_volterra_equation}, it holds in~\eqref{equ:our_volterra_solution} too.
\end{Lemma}

\begin{remark}\label{remark:convolution}
    The notation $\star$ is consistent with the usual convolution product, which corresponds to the special case where the values $a(t,s)$and $b(t,s)$ of the two Volterra kernels depend only on $t-s$. Making the identification $\Tilde{a} (t-s) \equiv a(t,s) \equiv a(t-s)$ and $\Tilde{b} (t-s) = b(t,s)$ the convolution product is recovered as $(\Tilde{a} \star \Tilde{b}) (t-s) \equiv (a \star b)(t,s)$. Likewise, the resolvent $A$ of the kernel $a$ equivalently expresses as function $A(t,s)= \Tilde{A}(t-s)= \sum_{i=1}^{\infty} a^{\star i}(t-s)$ of a single variable.
\end{remark}
\begin{remark}\label{remark:renewal}
    When $k(t,s)$ is convolutive, i.e there exists some integrable function $f \colon \R^+ \mapsto \R^p$ such that $k(t,s) = f(t-s)$, the Volterra equation~\eqref{equ:our_volterra_equation} is referred to as a \textit{renewal equation}. Taking $m$ as the restriction to $[0,T]$ of some $\Tilde{m} \colon \R \mapsto \R^p$, the domain of equation~\eqref{equ:our_volterra_equation} and its solution can then be extended from $t \in [0,T]$ to $t \in \R^+$, and the global solution is unique over $\R^+$ -- see Bacry \textit{et al.}~\cite[Lemma 3]{bacrylimit}. 
\end{remark}

Following remarks~\ref{remark:convolution} and~\ref{remark:renewal}
 we introduce a convolution-specific version of Lemma~\ref{Lemma:gripenberg}.

 \begin{Lemma}[ Lemma 3 in Bacry \textit{et al.} \cite{bacrylimit} or Lemmata 23-24 in Delattre \textit{et al.}~\cite{DelattreHawkesLargeNetworks}]\label{Lemma:renewal}
     Let $m \colon \R \mapsto [0,\infty)^p$ be a bounded function. Let $\phi \colon [0,\infty)] \mapsto \mathcal{M}_p(\R^+)$ be such that the matrix  $\int_0^{\infty} \phi(t) \dif t$ has spectral radius below $1$. Then,
     any function $x \colon \R^+ \mapsto \R^p$ such that\begin{equation}\label{equ:renewal_equation}
         x(t)
         \leq 
         m(t)
         +
         \int_0^t \phi(t-s) m(s) \dif s,
         \hspace{0.1cm}
         t \geq 0
     \end{equation}
     also verifies\begin{equation}\label{equ:renewal_solution}
         x(t) \leq 
         m(t) + 
         \int_0^t \psi(t-s) m(s) \dif s,
     \end{equation}
     with $\psi= \sum_{i=1}^{\infty}\phi^{\star i}$ well-defined in $L^1[0,\infty)$. If furthermore equality holds in~\eqref{equ:renewal_equation}, it holds in~\eqref{equ:renewal_solution} too.
 \end{Lemma}

 Accordingly, for any $x \in [0,1]$, we define 
\begin{equation*}
    \Psi(x,\cdot) \colon t \in [0,\infty) \mapsto   \Psi(x,t) :=\sum_{i=1}^{\infty} g(x)^i \varphi^{\star i}(t) \in \mathcal{M}_p(\R)
\end{equation*} the resolvent of the convolutive kernel $(s,t) \mapsto g(x) \varphi(t-s)$, where $\star$ is the convolution product, and we have made the equivalent identification to a function over $[0,\infty)$ detailed in Remarks~\ref{remark:convolution} and~\ref{remark:renewal}.  We also introduce the bounding resolvent
\begin{equation}\label{equ:bound_resolvent}
    \bar{\Psi} \colon t \in [0,\infty) \mapsto \sum_{i=1}^{\infty} \bar{\varphi}^{\star i}(t)
    =
    \sum_{i=1}^{\infty} \lVert g \rVert_{L^{\infty}}^i\varphi^{\star i}(t)
    \in \mathcal{M}_p(\R)
\end{equation}

 associated with the bounding kernel $(s,t) \mapsto \lVert g \rVert_{L^{\infty}} \varphi(t-s)$ defined in~\eqref{equ:bound_function}. It is easy to see from the expression of $K^T$ that for any $T>0$  and $(s,t) \in [0,T]^2$, one has the coordinate-wise inequality $K^T(t,s) \leq \Bar{\Psi}(t-s)$. The resolvent $K^T(t,s)$ thus inherits the integrability of $\varphi$ in the same fashion $\bar{\Psi}$ does, in that $\int_s^T K^T(t,s) \dif t \leq \int_0^{\infty} \Bar{\Psi}(t) \dif t < \infty$ coordinate-wise for any $0<s<T$. As the solutions of Lemmata~\ref{Lemma:gripenberg} and~\ref{Lemma:renewal} depend only on the values of the kernels and resolvents at $t>0$, one may equivalently consider $\varphi$ and the $\Psi(x,\cdot)$ as functions of $\R$ with support in $[0,\infty)$ -- this is consistent with our definition of $K^T(t,s)$ and allows for helpful simplifications in the notation. In the sequel, we will continuously resort to such extension without further comment. \\
 
 In all generality, the differentiability of $K^T(t,s)$ will also depend on the differentiability of $\varphi$. This is with the notable exception of the directional derivative along $(1,1)$ which depends on $g'$ only. This apparently anecdotal property will in fact repeatedly intervene in our proofs, as some helpful integration by parts formulas follow from Lemma~\ref{Lemma:IPP_base}.

 \begin{Lemma}\label{Lemma:IPP_base}
     Suppose $g$ verifies Assumption~\ref{ass:g_and_mu_are_C0}. Then, for any $n \in \mathbb{N}$ and any $T \geq t\geq s \geq 0$,
     \begin{equation*}
         (\partial_s + \partial_t) (k^T)^{\star i}
         (t,s)
         \leq 
         \frac{i}{T} \lVert g' \rVert_{L^\infty}
          \lVert g \rVert_{L^\infty}^{i-1} 
          \varphi^{\star i}(t-s).
     \end{equation*}
     In addition, $K^T(t,s)$ is then differentiable in the same direction and, for any $T \geq t\geq s \geq 0$,
     \begin{equation*}
         (\partial_s + \partial_t) 
         K^T(t,s)
         \leq 
         \frac{1}{T} \tilde{\Psi}(t-s),
     \end{equation*}
     where $\tilde{\Psi}(t) = \sum_{i=0}^{\infty} i \lVert g' \rVert_{L^\infty}
          \lVert g \rVert_{L^\infty}^{i-1} 
          \varphi^{\star k}(t)$ is well-defined in $L^1[0,\infty)$.
 \end{Lemma}

 \begin{proof}
     The proof proceeds by recursion. At $n=1$, for any $t \geq s$ and any $h >0$, following some immediate simplification,
     \begin{equation*}
         \frac{1}{h} \{ k^T(t+h,s+h) - k^T(t,s) \}
         =
         \frac{1}{h} \{
            g( \frac{t+h}{T})
            -
            g(\frac{t}{T}) \}
            \varphi(t-s),
     \end{equation*}
     and the property appears naturally as $h \to 0$. Let $n \geq 1$, and suppose the property holds up to $n$. Then, following the change of variable $u \mapsto u+h$, for any $T \geq t \geq s \geq 0$ 
     \begin{align*}
         \frac{1}{h}
         \{ (k^T)^{\star (n+1)}(t+h,s+h)
         &-
         (k^T)^{\star (n+1)}(t,s)
         \}
         =
         \\
         &\int_s^t
         \frac{1}{h}
         \{
         g(\frac{t+h}{T})-g(\frac{t}{T})
         \}
         \varphi(t-u) (k^T)^{\star n}(u+h,s+h) \dif u
         \\
         +
         &\int_s^t
         g(\frac{t}{T})
         \varphi(t-u) 
         \frac{1}{h}
         \{
         (k^T)^{\star n}(u+h,s+h)
         -
         (k^T)^{\star n}(u,s)
         \}
         \dif u.
     \end{align*}
     Letting $h \to 0$ this yields
     \begin{equation*}
         (\partial_t + \partial_s)  (k^T)^{\star(n+1)}
         (t,s) 
         =
         \frac{1}{T}g'(\frac{t}{T}) 
         \int_s^t \varphi(t-u) k^{\star n}(u,s) \dif u
         +
         \int_s^t 
         g(\frac{t}{T})
         \varphi(t-u)
          (\partial_t + \partial_s) 
          (k^T)^{\star n}(u,s) \dif u.
     \end{equation*}
     Using that $k^T(t,s) \leq \bar{\varphi}(t-s)=\lVert g \rVert_{L^{\infty}} \varphi(t-s)$ and our recursion hypothesis, this is less than
     \begin{equation*}
        \frac{1}{T}
         \lVert g' \rVert_{L^{\infty}}
         \lVert g \rVert_{L^{\infty}}^{n}
         \varphi^{\star (n+1)}(t-s)
         +
         \frac{1}{T}
         n \lVert g' \rVert_{L^{\infty}}
         \lVert g \rVert_{L^{\infty}}^n
         \varphi^{\star (n+1)}(t-s),
     \end{equation*}
     which is the required bound. If $\lVert g \rVert_{L^{\infty}}=0$, there is nothing else to say. Otherwise, for any $n \in \mathbb{N}^{\star}$,
     \begin{align*}
         \int_0^{\infty}  \lVert g' \rVert_{L^{\infty}}
         \lVert g \rVert_{L^{\infty}}^n
         \varphi^{\star (n+1)}(t)
         \dif t
         =
         \int_0^{\infty}  \frac{\lVert g' \rVert_{L^{\infty}}}{\lVert g \rVert_{L^{\infty}}}
         \lVert g \rVert_{L^{\infty}}^{n+1}
         \varphi^{\star (n+1)}(t)
         \dif t
         =
         \frac{\lVert g' \rVert_{L^{\infty}}}{\lVert g \rVert_{L^{\infty}}}
         \Big( \int_0^{\infty} \bar{\varphi} (t) \dif t \Big)^{n+1},
     \end{align*}
     and $\sum_{i=1}^{n} i \lVert g' \rVert_{L^\infty}
          \lVert g \rVert_{L^\infty}^{i-1} 
          \varphi^{\star i}(t)$ converges in  $L^1[0,\infty)$ for the same reasons as $\bar{\Psi}$.
 \end{proof}

 The behaviour of the Volterra equation~\eqref{equ:our_volterra_equation} is now clearer and one may see that some of the key properties of the convolutive case~\eqref{equ:renewal_equation} extend to our problem. However, equation~\eqref{equ:our_volterra_equation} differs fundamentally from~\eqref{equ:renewal_equation} in that the solutions of the latter are globally defined over $[0,\infty)$, whereas the resolvent of the former is in a local nature and depends on $T$.\\
 
 As a consequence, asymptotic estimates do not come as easily. Letting $(s,x) \in [0,T] \times [0,1]$ for instance, $\int_s^T K^T(t,s) \dif t$ features no apparent limit as $T \to \infty$ contrarily to $\int_s^T \Psi(x,t-s) \dif t$. In spite of these difficulties, we are still able to establish the convergence of $K^T(t,s)$ in a Cesàro sense. 

\begin{proposition}\label{prop:deterministic_convergence}
    Let $m \colon [0,1] \mapsto \R^p$ be a continuous function. Under Assumptions~\ref{ass:g_and_mu_are_C0} to~\ref{ass:stability}, 
    \begin{equation*}
        \sup_{u \in [0,1]}
        \Big\lVert
        \frac{1}{T}
        \int_0^{Tu}
            \int_s^{Tu} K^T(t,s)
            m\big( \frac{s}{T} \big)
            \dif t
        \dif s
        -
        \int_0^u \Big( \int_0^{\infty} \Psi(x,s) \dif s \Big) m(x)  \dif x\Big\rVert
        \to
        0
    \hspace{0.1cm}
    \textup{ as }
    \hspace{0.1cm}
    T \to \infty .
    \end{equation*}
    If furthermore $m$ is Lipschitz-continuous and Assumption~\ref{ass:sqrt_integ} holds, the convergence occurs at rate $o(\frac{1}{\sqrt{T}})$.
\end{proposition}
The preceding technical result is seemingly new in the theory of Volterra integral equations, and perhaps of independent interest. The next sub-section is dedicated to its proof.

\subsection{Proof of Proposition~\ref{prop:deterministic_convergence}}
Our proof relies on an approximation scheme reminiscent of the broad principle behind the strategy of Kwan \textit{et al.}~\cite{Kwan1}~\cite{Kwan2}. Namely, we intend to approximate the Volterra resolvent $(s,t) \mapsto K^T(t,s)$ by local convolutive counterparts on sub-intervals of $[0,Tu]$, which lengths are arbitrarily large in absolute terms while arbitrarily small relatively to $T$.  Accordingly let $\{ \Delta_T \}$ be a collection of real numbers such that $\Delta_T \to \infty$ and $\Delta_T = o(T)$ as $T \to \infty$. Without loss of generality, we ask furthermore that $\frac{T}{\Delta_T} \in \mathbb{N}^{\star}$ so as to avoid adressing inessential edge cases.  One thus obtains a partition of $[0,Tu]$.

\begin{Lemma}\label{lemma:deterministic_approximation}
Under Assumptions~\ref{ass:g_and_mu_are_C0} and~\ref{ass:stability}, there is some constant $C_{g,\varphi}$ depending only on $\varphi$ and $g$ such that, for any $T>0$, $s \in [0,T]$ and $u \in [0,1]$,
    \begin{equation*}
    \Big\lVert 
    \int_{s}^{Tu}
        K^T(t,s) 
    \dif t
    -
    \sum_{i=0}^{\nicefrac{T}{\Delta_T}-1}
        \int_{ i\Delta_T u}^{(i+1) \Delta_T u}
         \Psi(i u \frac{ \Delta_T }{T},t-s)  \dif t
    \Big\rVert
    \leq
    C_{g,\varphi}
    \omega(g,\frac{\Delta_T}{T})
        \end{equation*}
\end{Lemma}
\begin{proof}
    Let $s \in [0,T]$, $u \in [0,1]$. For any $t \in [s,T]$ and any $i \in \mathbb{N}$ such that $i u \Delta_T \leq t \leq (i+1) u\Delta_T$,
    \begin{align*} 
            K^T(t,s)
            -
            \Psi(i u \frac{\Delta_T}{T}, t-s) 
        &=
        \big( g( \frac{t}{T})  - g( i u  \frac{\Delta_T}{T})\big) \varphi(t-s)
        \\
        &+ 
        \big( g( \frac{t}{T})  - g( i u \frac{\Delta_T}{T})\big)  
        \int_s^t
        \varphi(t-u)
        K^T(u,s) \dif u 
        \\
        &+
        g(i u\frac{\Delta_T}{T})
            \int_s^t 
                \varphi(t-u)
                \Big(
                K^T(u,s)
                -
                \Psi(i \frac{\Delta_T}{T}, u-s)
                \Big)
            \dif u.
    \end{align*}
    Hence, using the fact that $K^T(t,s) \leq \Bar{\Psi}(t-s)$, one obtains the coordinate-wise inequality
    \begin{align}
            K^T(t,s)
            -
            \Psi(i u \frac{\Delta_T}{T}, t-s)
            &\leq 
            \omega(g,\frac{\Delta_T}{T})
            \big(  \varphi + \varphi \star \Bar{\Psi} \big) (t-s)
            \label{equ:renewal_bound}
            \\
            &+
            \int_s^t \Bar{\varphi}(t-u)
            \Big(
                K^T(u,s)
                -
                \Psi(i u \frac{\Delta_T}{T}, u-s)
                \Big)
            \dif u.
            \nonumber
    \end{align}
     Defining $\Bar{\Bar{\Psi}}= \varphi + \varphi \star \Bar{\Psi}$, we have obtained in ~\eqref{equ:renewal_bound} a renewal inequation with kernel $\Bar{\varphi}$. It resolves into the coordinate-wise inequality
    \begin{equation*}
        K^T(t,s) - \Psi(i u \Delta_T,t-s)
        \leq 
        \omega(g,\frac{\Delta_T}{T})
        \Big(
        \Bar{\Bar{\Psi}}
        +
        \Bar{\Psi}
        \star 
        \Bar{\Bar{\Psi}}
        \Big)
        (t-s).
    \end{equation*}
   Since $\rho( \int_0^{\infty} \Bar{\varphi}(t) \dif t)<1$, $\Bar{\Psi}$ and thus $\Bar{\Bar{\Psi}}$ are integrable (see for instance Gripenberg~\cite[Theorem 4.3 in Chapter 2]{gripenberg}) . Integrating over $t$ in each $A^k_{T,u}$ and summing back the resulting bounds yields the desired inequality with the constant
    \begin{equation*}
        C_{g,\varphi}
        =
        \Big\lVert 
        \int_0^{\infty}
            \Bar{\Bar{\Psi}}(t)
            +
            \Bar{\Psi}
            \star
            \Bar{\Bar{\Psi}}
            (t)
        \dif t
        \Big\rVert_1.
    \end{equation*}
\end{proof}

Note that as $K^T(t,s)=\Psi(x,t-s)=0$ when $s>t$, the inequality~\eqref{equ:renewal_bound} is trivial if $i \Delta_T u + \Delta_T u < s$ and we need not discuss the position of the $A_i^{T,u}$ relatively to $s$ in the statement or in the proof of Lemma~\ref{lemma:deterministic_approximation}. A direct adaptation of the proof of Lemma~\ref{lemma:deterministic_approximation} yields the useful continuity result below.

\begin{Lemma}\label{Lemma:gamma_continuity}
    The family of functions $(x \in [0,1] \mapsto \int_s^{\infty} \Psi(x,t) \dif t )_{s \geq 0}$ is uniformly equi-continuous in $x$.
\end{Lemma}

\begin{proof}
    Re-arranging the terms and using that $\Psi(x,\cdot) \leq \Bar{\Psi}$ and $g(y) \varphi \leq \Bar{\varphi}$, one has for any $s>0$, $t>0$, and any $x,y \in [0,1]$,
    \begin{equation*}
        \Psi(x,t)
        -
        \Psi(y,t)
        \leq
        \omega(g,\lvert x-y\rvert)
        \big(
            \varphi
            +
            \varphi \star \Bar{\Psi}
        \big)(t)
        +
        \int_0^t \Bar{\varphi}(t-u) 
        ( \Psi(x,u)
        -
        \Psi(y,u)
        )
        \dif u,
    \end{equation*}
    where the renewal inequality above holds coordinate-wise. The same arguments as for the proof of Lemma~\ref{lemma:deterministic_approximation} then show
    \begin{equation*}
        \int_s^{\infty} \lVert \Psi(x,t)
        -
        \Psi(y,t)
        \rVert_1 \dif t
        \leq 
        \omega(g,\varepsilon)
        \int_0^{\infty} 
        \lVert 
            \varphi(t)  
            +
            2 \varphi \star \Bar{\Psi}
            +
            \varphi \star
            \Bar{\Psi}^{\star 2}
        \rVert_1 
        \dif t,
    \end{equation*}
    and the Lemma deduces from the (uniform) continuity of $g$ over the compact $[0,1]$.
\end{proof}

\begin{Lemma}\label{Lemma:second_deterministic_approximation}
    Under Assumptions~\ref{ass:g_and_mu_are_C0} and~\ref{ass:stability},
    \begin{equation*}
    \frac{1}{T}
    \rVert 
        \int_0^{Tu}
            \int_0^t 
                K^T(t,s) 
                m \big( \frac{s}{T} \big) 
            \dif s
        \dif t
        -
        \int_0^{Tu} 
            \sum_{i=0}^{\nicefrac{T}{\Delta_T}-1}
            \int
            _{i \Delta_T u}
            ^{(i+1) \Delta_Tu}
                \Psi(i u \frac{\Delta_T}{T},t-s) 
                m\big( \frac{s}{T} \big) 
            \dif t
        \dif s 
    \rVert_1
    \to
    0
    \end{equation*}
    as $T \to \infty$, uniformly in $u \in [0,1]$.
\end{Lemma}
\begin{proof}
    For any $T>0$ and any $u \in [0,1]$, by Fubini's Theorem,
    \begin{align*}
        \frac{1}{T}
        \int_0^{Tu} 
            \int_0^t 
                K^T(t,s) 
                m\big( \frac{s}{T} \big)
            \dif s
        \dif t
        =
        \frac{1}{T}
        \int_0^{Tu}
            \Big( 
            \int_s^{Tu}
                K^T(t,s) 
            \dif t
            \Big)
            m\big( \frac{s}{T} \big)
        \dif s,
    \end{align*}
    hence, using Lemma~\ref{lemma:deterministic_approximation}, one has for any $u \in [0,1]$
    \begin{align*}
         \frac{1}{T}
         \Big\lVert 
        \int_0^{Tu}
            \int_0^t 
                K^T(t,s) 
                m \big( \frac{s}{T} \big) 
            \dif s
        \dif t
        -
        \frac{1}{T}
        \int_0^{Tu}
        \sum_{i=0}^{\nicefrac{T}{\Delta_T}-1}
            &\int
            _{i \Delta_T u}
            ^{(i+1) \Delta_T u }
                \Psi(i u \frac{\Delta_T}{T},t-s) 
                m\big( \frac{s}{T} \big) 
            \dif t
        \dif s 
        \Big\rVert_1
        \\
        &\leq
        C \sup_{x \in [0,1]} \lVert m(x) \rVert_1
        \frac{1}{T} 
        \int_0^T
        \sum_{i=0}^{ \nicefrac{T}{\Delta_T}-1 }
        u \frac{\Delta_T}{T} 
            \omega(g,\frac{\Delta}{T})
        \dif t 
        \\
        &\leq C' \omega(g , \frac{\Delta_T}{T})
        \xrightarrow[T \to \infty]{}
        0,
    \end{align*}
    where $C$ and $C'$ are positive constants. 
\end{proof}

\begin{Lemma}\label{ref:third_deterministic_approximation}Under Assumptions~\ref{ass:g_and_mu_are_C0} to~\ref{ass:stability}, as $T \to \infty$,
    \begin{equation*}
    \Big\lVert
        \frac{1}{T}
        \int_0^{Tu} 
        \sum_{i=0}^{ \nicefrac{T}{\Delta_T}-1 }
            \int_{i \Delta_T u}
            ^{ (i+1) \Delta_Tu}
                \Psi(i u \frac{\Delta_T}{T},t-s) 
                m\big( \frac{s}{T} \big) 
            \dif t
        \dif s
        -
        \int_0^u
            \Big( 
            \int_0^{\infty}
                \Psi(x,s) 
            \dif s
            \Big)m(x) 
        \dif x 
        \Big\rVert_1
        \to
        0,
    \end{equation*}
    uniformly in $ u\in [0,1]$
\end{Lemma}
\begin{proof}
    For any $T>0$,  $u \in [0,1]$, and any $i \in \mathbb{N}$, by Fubini's Theorem and the fact that $\Psi(x,t-s)=0$ for any $s>t$,
    \begin{align*}
         \int_0^{Tu} 
            \int_{ i \Delta_T u}^{(i+1) \Delta_T u }
                &\Psi(i u \frac{\Delta_T}{T},t-s) m\big( \frac{s}{T} \big)
            \dif t 
        \dif s
       =
        \int_{i \Delta_T u}^{(i+1) \Delta_T u  } 
            \int_0^t 
                \Psi(i u \frac{\Delta_T}{T},s)
                m\big( 
                \frac{t}{T} - \frac{s}{T}
            \big)
            \dif s 
        \dif t.
    \end{align*}
    Note that, for any $u \in [0,1]$, the first term at $i=0$ abides by the bound 
    \begin{equation*}
        \big\lVert
        \frac{1}{T}
        \int_0^{u \Delta_T} \int_0^t \Psi(0,s) m\big( \frac{t}{T}-\frac{s}{T}\big)\dif s \dif t
        \big\rVert_1 
        \leq 
        \frac{\Delta_T}{T} \sup_{x \in [0,1]} \lVert m(x) \rVert \Big( \int_0^{\infty} \lVert \bar{\Psi}(s) \rVert_1 \dif s \Big),
    \end{equation*}
    which vanishes as $T \to \infty$, and one may thus consider exclusively the summands wherein $i \geq 1$. Working \textit{à la} Cesàro, we let $\eta \in (0,1]$. For any $u \in [0,\eta)$,
    \begin{align*}
        \frac{1}{T}
        \Big\lVert
        \sum_{i=1}^{\nicefrac{T}{\Delta_T} -1}
        \int_{i \Delta_T u}^{(i+1) \Delta_T u }
        \int_0^t \Psi( i u \frac{\Delta_T}{T},s) \dif s \dif t
        \Big\rVert_1
       \leq  \eta
         \Big( 
            \int_0^{\infty}
                \lVert \Bar{\Psi}(t) \rVert_1 
            \dif t
        \Big),
    \end{align*}
    and a similar bound holds for the limit term $\int_0^u (\int_0^{\infty} \Psi(x,s) \dif s) m(x)\dif x$; while for any $u \in [\eta,1]$, $i \in \mathbb{N}^{\star}$, and any $t>0$ such that $u i \Delta_T \leq t \leq u (i+1) \Delta_T$
    \begin{align}
    \int_0^t 
        \Psi(i u\frac{\Delta_T}{T},s) 
        m \big(\frac{t}{T}- \frac{s}{T} \big) 
    \dif s
    &-
    \Big( 
    \int_0^{\infty} 
        \Psi(i u\frac{\Delta_T}{T},s)
    \dif s \Big) 
    m\big( i u \frac{\Delta_T}{T} \big) 
    \nonumber
    \\
    &=
        \int_0^t 
            \Psi(i u\frac{\Delta_T}{T},s) 
            \big\{
            m\big( 
                \frac{t}{T} - \frac{s}{T}
            \big)
            -
            m\big(  i u \frac{\Delta_T}{T} \big)
            \big\}
        \dif s 
        \label{equ:Cesaro_sep}
    \\
       &-
        \Big( 
        \int_t^{\infty}
            \Psi(i u\frac{\Delta_T}{T},s) 
        \dif s
        \Big)
        m\big( i u \frac{\Delta_T}{T} \big).
        \nonumber
    \end{align}
    Separating the first integral on the right-hand side of the preceding equation about $\eta \Delta_T$, one finds that it is bounded in $\lVert \cdot \rVert_1$-norm by
    \begin{align*}
        \Big( 
        \int_0^{\eta \Delta_T}
            \lVert \Psi(i u \frac{\Delta_T}{T},s) \rVert_1
        \dif s \Big)
        \omega(m, \frac{\Delta_T}{T})
        +
        2
         \Big( 
        \int_{\eta \Delta_T}^{\infty}
            \lVert \Psi(i u \frac{\Delta_T}{T},s) \rVert_1
        \dif s \Big) 
        \lVert m \rVert_{L^{\infty}}.
    \end{align*}
    The last term in~\eqref{equ:Cesaro_sep} is also bounded by $
        ( 
        \int_{\eta \Delta_T}^{\infty} 
            \lVert 
            \Bar{\Psi}(s)
            \rVert_1 
        \dif s) 
         \lVert m \rVert_{L^{\infty}}
    $. Consequently, for any $u \in [\eta,1]$,
    \begin{align*}
        \frac{1}{T}
        \sum_{i=1}^{\nicefrac{T}{\Delta_T} -1}
        \int_{i \Delta_T u}^{(i+1) \Delta_T u } 
            \int_0^t 
                \Psi(i u \frac{\Delta_T}{T},t-s) 
                m\big( \frac{s}{T}\big)
            &\dif s 
        \dif t
        -
         \Big( 
            \int_0^{\infty} 
                \Psi(i u \frac{\Delta_T}{T},s) 
            \dif s \Big) 
            m( i \frac{\Delta_T}{T} u)
        \dif t
        \\
        &\leq 
        C\Big(
        \frac{ \Delta_T}{T}
        \sum_{i=1}^{ \nicefrac{T}{\Delta_T} -1} 
        \omega(m,\frac{\Delta_T}{T})
        +
        \int_{\eta \Delta_T}^{\infty} \lVert \Bar{\Psi}(s)  \rVert
        \dif s
        \Big)
        \\
        &\leq 
        C
        \Big( 
        \omega(m,\frac{\Delta_T}{T})
        +
        \int_{\eta \Delta_T}^{\infty} \lVert \Bar{\Psi}(s) \rVert_1 \dif s
        \Big),
    \end{align*}
    where $C$ is some positive constant. The bound goes to $0$ on account of the integrability of $\Bar{\Psi}$ and uniform continuity of $m$ over the compact interval $[0,1]$. Altogether, we have established the uniform convergence of the expression of interest towards
    \begin{equation*}
     \frac{1}{T}
        \sum_{i=0}^{ \nicefrac{T}{\Delta_T} -1}
        \int_{i \Delta_T u}^{(i+1) \Delta_T u }      
        \Big( 
            \int_0^{\infty}
                \Psi(i u \frac{\Delta_T}{T},s) 
            \dif s 
        \Big) 
        m\big( i u \frac{\Delta_T}{T}\big)
        \dif t,
    \end{equation*}
    wherein one recognises a Riemann series. Using the change of variable $x \mapsto ux$ then,
    \begin{equation}\label{equ:Riemann_convergence}
    \sum_{i=0}^{ \nicefrac{T}{\Delta_T} -1}
        u\frac{\Delta_T}{T}
        \Big( 
            \int_0^{\infty}
                \Psi( i u \frac{\Delta_T}{T},s) 
            \dif s 
        \Big) m( i \frac{\Delta_T}{T} u)
        \to
         \int_0^u \Big( \int_0^{\infty}\Psi(x,s) \dif s\Big) m(x) \dif x
    \end{equation}
    
    as $T \to \infty$. This holds uniformly in $u$ by virtue of the uniform equicontinuity in $u$  of the Riemann series over the compact interval $[0,1]$, itself a consequence of the uniform continuity of $x \mapsto \int_0^{\infty} \Psi(x,s) \dif s$ obtained in Lemma~\ref{Lemma:gamma_continuity}.  The Lemma ensues.
\end{proof}

Gathering Lemmata~\ref{lemma:deterministic_approximation} to~\ref{ref:third_deterministic_approximation}, we have established the first assertion of Proposition~\ref{prop:deterministic_convergence}.  
When assumption~\ref{ass:sqrt_integ} holds, the rate of convergence may be further quantified.

\begin{Lemma}
   Suppose $m$ is Lipschitz-continuous and Assumptions~\ref{ass:g_and_mu_are_C0} to~\ref{ass:stability} and ~\ref{ass:sqrt_integ} hold. Then, the convergences in Lemmata~\ref{prop:deterministic_convergence} to~\ref{ref:third_deterministic_approximation} occur at rate $o(\frac{1}{\sqrt{T}})$
\end{Lemma}

\begin{proof}
    The final bounds obtained in the proofs of Lemmata~\ref{lemma:deterministic_approximation} and~\ref{Lemma:second_deterministic_approximation} decrease at rate $\omega(g,\frac{\Delta_T}{T})$ which is in $ \mathcal{O}(\frac{\Delta_T}{T})$ since $g$ is continuously differentiable over $[0,1]$ under Assumption~\ref{ass:g_and_mu_are_C0}, hence Lipschitz-continuous. As we have only asked that $\Delta_T = o(T)$, it is not restrictive to further assume  that $\Delta_T = o(\sqrt{T})$, in which case a rate of $o(\frac{1}{\sqrt{T}})$ is indeed achieved.\\
    
    Obtaining an explicit convergence rate in Lemma~\ref{ref:third_deterministic_approximation} is slightly more involved, and we need to revisit its proof. Suppose $\varphi$ satisfies Assumption~\ref{ass:sqrt_integ}, in which case $\int_0^{\infty} \lVert \bar{\Psi}(s) \rVert_1 \sqrt{s} \dif s< \infty$ (see for instance Bacry \textit{et al.}~\cite[Proof of Lemma 5]{bacrylimit}).  Let $m \colon [0,1] \mapsto \R$ be some Lipschitz-continuous function, with $\lVert m(y)-m(x) \rVert \leq C \lvert x-y\rvert$.  Firstly, for any $T>0$ and any $u \in [0,1]$, the difference
    \begin{equation*}
        \frac{1}{\sqrt{T}}
        \sum_{i=0}^{\lfloor \nicefrac{T}{\Delta_T} \rfloor}
        \int_{i u\Delta_T }^{(i+1)u\Delta_T}
        \int_0^t \Psi(i u \frac{\Delta_T}{T},s) 
        m\big(\frac{t}{T}-\frac{s}{T} \big) \dif s
        -
        \int_{i u\Delta_T }^{(i+1)u\Delta_T}
        \Big(
        \int_0^t \Psi(i u \frac{\Delta_T}{T},s) 
        \dif s 
        \Big) 
        m\big(\frac{t}{T} \big)  \dif t
    \end{equation*}
    is bounded in $\lVert \cdot \rVert_1$-norm by
    \begin{align}
        \frac{1}{\sqrt{T}}
        \sum_{i=0}^{\lfloor \nicefrac{T}{\Delta_T} \rfloor}
         \int_{i u\Delta_T }^{(i+1)u\Delta_T}
        \int_0^t &\lVert \Psi(i u \frac{\Delta_T}{T},s)
        \rVert_1
         \lVert  m\big(\frac{t}{T}-\frac{s}{T} \big) -  m\big(\frac{t}{T} \big)\rVert_1  
         \dif s
         \dif t
         \nonumber
         \\
         &\leq
         \frac{C}{\sqrt{T}}
        \sum_{i=0}^{\lfloor \nicefrac{T}{\Delta_T} \rfloor}
         \int_{i u\Delta_T }^{(i+1)u\Delta_T}
         \Big(
        \int_0^t \lVert \Psi(i u \frac{\Delta_T}{T},s)
        \rVert_1
         \frac{s}{T}  
         \dif s
         \Big) 
         \dif t
         \label{equ:ready_IPP_bound}
         \\
         & \leq 
         \frac{C}{\sqrt{T}} \int_0^T  
         s\lVert \bar{\Psi}(s)
        \rVert_1 \dif s.
        \nonumber
    \end{align}

    It is an elementary exercise to see that $\frac{1}{\sqrt{T}} \int_0^T s \lVert \bar{\Psi}(s) \rVert_1 \dif s \to 0$  as $T \to \infty$ if $\int_0^{\infty}\lVert \bar{\Psi} (s)\rVert_1 \sqrt{s} <\infty$  (see again Bacry \textit{et al.}~\cite[Proof of Lemma 5]{bacrylimit}). The bound above then vanishes as $T \to \infty$. Secondly, using once more the Lipschitz-continuity of $m$, for any $T>0$ and any $u \in [0,1]$, 
    \begin{equation*}
    \big\lVert
         \frac{1}{\sqrt{T}}
        \sum_{i=1}^{\lfloor \nicefrac{T}{\Delta_T} \rfloor}
        \int_{i u\Delta_T }^{(i+1)u\Delta_T}
        \Big(
        \int_0^t \Psi(i u \frac{\Delta_T}{T},s) 
        \dif s 
        \Big) 
        \{
        m\big(\frac{t}{T} \big)  
        -
        m\big( i u\frac{\Delta_T}{T} \big)
        \}
        \dif t
    \big\rVert_1
    \end{equation*}
    is bounded by $\frac{C}{\sqrt{T}} \int_0^{\infty} \lVert \bar{\Psi}(s) \rVert_1 \dif s$ and goes to $0$ as $T \to \infty$. Third and finally, 
    \begin{align*}
         \frac{1}{\sqrt{T}}
        \sum_{i=0}^{\lfloor \nicefrac{T}{\Delta_T} \rfloor}
        \int_{i u\Delta_T }^{(i+1)u\Delta_T}
        \Big(
        &\int_0^t \Psi(i u \frac{\Delta_T}{T},s) 
        \dif s 
        -
        \int_0^{\infty} \Psi(i u \frac{\Delta_T}{T},s) 
        \dif s 
        \Big) 
        m\big( i u\frac{\Delta_T}{T}  \big)  
        \dif t
        )
        \\
        &\leq
        \sup_{x \in [0,1]} \lVert m(x) \rVert_1
        \frac{1}{\sqrt{T}}
        \sum_{i=0}^{\lfloor \nicefrac{T}{\Delta_T} \rfloor}
        \int_{i u\Delta_T }^{(i+1)u\Delta_T}
        \Big(
        \int_t^{\infty} 
        \lVert \bar{\Psi}(s) 
        \rVert_1 \dif s
        \Big)
        \dif t.
    \end{align*}
   Recall that we work under $\Delta_T=o(\sqrt{T})$. The first term at $i=0$ obeys the bound
   $  \frac{\Delta_T}{\sqrt{T}}    \int_0^{\infty}  \lVert \bar{\Psi}(t) \rVert_1 \dif t
    $,
   and one may thus discard it as we have done in Lemma~\ref{ref:third_deterministic_approximation}. Hence we consider only the summands at $i \geq 1$. Using the bound $ \int_t^{\infty} \lVert \bar{\Psi} \rVert_1 \dif s \leq  t^{-\nicefrac{1}{2}} \int_t^{\infty}  \lVert \sqrt{s} \bar{\Psi} (s)\rVert_1 \dif s$ then, the expression above is dominated (up to its first term and some multiplicative constant) by
    \begin{align*}
        \frac{1}{\sqrt{T}}\sum_{i=1}^{\lfloor\nicefrac{T}{\Delta_T} \rfloor}
        \Big( \int_{i u\Delta_T }^{(i+1)u\Delta_T}
        \frac{1}{\sqrt{t}} \dif t \Big)
        \Big( 
        &\int_{i \Delta_T}^{\infty} \lVert \Bar{\Psi}(s) \rVert_1
        \sqrt{s}\dif s \Big)
        \\
        &=
        \frac{2}{\sqrt{T}}
        \sum_{i=1}^{\lfloor\nicefrac{T}{\Delta_T} \rfloor}
        \Big( \sqrt{(i+1) u \Delta_T} - \sqrt{i u \Delta_T}\Big)
        \Big( 
        \int_{iu \Delta_T}^{\infty} \lVert \Bar{\Psi}(s) \rVert_1
        \sqrt{s}
        \dif s \Big)
        \\
        &\leq 
        \frac{2}{\sqrt{T}}\Big( \sqrt{uT} - \sqrt{u\Delta_T} \Big)
        \int_{u \Delta_T}^{\infty} \lVert \bar{\Psi}(s) \rVert_1 
        \sqrt{s}\dif s.
    \end{align*}
    Let $\eta \in [0,1]$ and separate the cases $u \in [0,\eta)$ and $u \in [\eta,1)$. The bound above is then itself dominated either by $2\eta \int_0^{\infty} \lVert \bar{\Psi}(s) \rVert_1 \dif s$ or by $2\int_{\eta \Delta_T}^{\infty} \lVert \bar{\Psi}(s) \rVert_1 \dif s$. It may thus be taken arbitrarily small as $T \to \infty$ on account of the integrability of $\bar{\Psi}$. There remains only to verify that the convergence of the Riemann series~\eqref{equ:Riemann_convergence} does occur at rate $o(\frac{1}{\sqrt{T}})$ too. However, some straightforward calculations and Lemma~\ref{Lemma:gamma_continuity} show that
    \begin{align*}
    \big\lVert
        \sum_{i=0}^{\nicefrac{T}{\Delta_T}-1} u \frac{\Delta_T}{T}( \int_0^{\infty} \Psi(i u \frac{\Delta_T}{T},s) \dif s) m(i u \frac{\Delta_T}{T})
        -
        \int_{i \Delta_T  }^{(i+1) \Delta_T}
       u \Big( \int_0^{\infty} \Psi(u\frac{x}{T},s) \dif s \Big)
        m\big( u\frac{x}{T} \big) \dif x
        \big\rVert_1
    \end{align*}
    is dominated -- up to some multiplicative constant -- by $\omega(g,\nicefrac{\Delta_T}{T}) + \omega(m,\nicefrac{\Delta_T}{T})$.  By Lipschitz-continuity of $g$ and $m$, this bound is of order $o(\frac{\Delta_T}{T})$, which is the required estimate as we work under $\Delta_T= o(\sqrt{T})$.
\end{proof}

Before we proceed to the proofs of our main results, we introduce some stochastic counterparts to the Volterra theory we have presented so far.

\subsection{A \textit{trend plus noise} decomposition of the intensity}

\begin{Lemma}[An extension of Jaisson \& Rosenbaum~\cite{JaissonRosenbaum}]\label{Lemma:Jaisson_like}
    Under Assumptions~\ref{ass:g_and_mu_are_C0} to~\ref{ass:stability}, for any $t \leq T$, 
    \begin{equation}\label{equ:jaisson_non_linear_equ}
        \lambda^T_t
        =
        \mu \big( \frac{t}{T} \big)
        +
        \int_0^t
            K^T(t,s) 
            \mu \big( \frac{s}{T} \big) 
        \dif s
        +
            \int_0^t K^T(t,s) \dif M^T_s.
    \end{equation}
\end{Lemma}
\begin{proof}
    For any $0 \leq t \leq T$, separating the martingale and finite variation parts in $\dif N^T_s$, 
    \begin{align*}
    \lambda_t^T =
    \mu(\frac{t}{T})
    +
    \int_0^t
        k^T(t,s) 
    \dif M_s^T 
    +
    \int_0^t 
        k^T(t,s) 
        \lambda^T_s 
    \dif s.
    \end{align*}
    One recognises hereinabove a Volterra equation in $(\lambda^T_t)$ with kernel $k^T$.  Applying Lemma~\ref{Lemma:gripenberg},
    \begin{align*}
        \lambda^T_t 
        =
        \mu(\frac{t}{T})
        +
        \int_0^t
            K^T(t,s) \mu(\frac{s}{T}) 
        \dif s
        +
        \int_0^t 
            k^T(t,s) 
        \dif M_s^T 
        +
        \int_0^t
            K^T(t,s)
            \int_0^s
                k^T(s,u)
            \dif M^T_u
        \dif t.
    \end{align*}
    By Fubini's Theorem, for any $ 0 \leq t \leq T$,
    \begin{equation*}
        \int_0^t 
        K^T(t,s) 
            \int_0^s 
                k^T(s,u) 
            \dif M^T_u
        \dif s
        =
        \int_0^t
        \int_u^t
        K^T(t,s)
        k^T(s,u) \dif s
        \dif M^T_u
        =
        \int_0^t 
            (K^T \star k^T) (t,u)  
        \dif M^T_u,
    \end{equation*}
    and the Lemma follows from the fact that $k^T + K^T \star k^T = K^T$.
\end{proof}
Lemma~\ref{Lemma:Jaisson_like} yields the two immediate useful corollaries~\ref{coro:mean_intensity} \&~\ref{coro:intensity_bound}.
 \begin{corollary}\label{coro:mean_intensity}
     Under Assumptions~\ref{ass:g_and_mu_are_C0} to~\ref{ass:stability}, for any $T>0$ and any $t \in [0,T]$,
     \begin{equation*}
        \mathbb{E}[
            \lambda^T_t
        ]
        =
            \mu \big( \frac{t}{T} \big)
        +
        \int_0^t 
            K^T(t,s)
            \mu \big( \frac{s}{T} \big)
        \dif s.
     \end{equation*}  
 \end{corollary}

 Recall that for any $T>0$ and $(s,t) \in [0,T]^2$, one has $K^T(t,s) \leq \Bar{\Psi}(t-s)$ with the bounding resolvent $\Bar{\Psi} \in L^1[0,\infty)$ defined in~\eqref{equ:bound_resolvent}. Furthermore, as a geometric series, $\boldsymbol{Id} + \int_0^{\infty} \bar{\Psi}(t) \dif t$  simplifies to $( \boldsymbol{Id} - \int_0^{\infty} \bar{\varphi} (t) \dif t)^{-1}$. Hence the result hereinbelow.

\begin{corollary}\label{coro:intensity_bound}
     Under Assumptions~\ref{ass:g_and_mu_are_C0} to~\ref{ass:stability}, for any $0 \leq t \leq T$, one has the coordinate-wise inequality
     \begin{equation*}
         \mathbb{E}\big[  
            \lambda^T_t
            \big]
         \leq 
          \sup_{x \in [0,1]} \lVert \mu(x) \rVert_1 
         \Big(\boldsymbol{Id} - \int_0^{\infty}
         \bar{\varphi}(t) \dif t\Big)^{-1}.
     \end{equation*}
\end{corollary}

\section{Proofs for section~\ref{section:limit_theorems}}\label{section:proof_FTCL}
 We work under the setting of section~\ref{section:limit_theorems}.  Let us introduce the family of processes $(P^T_u)_{u \in [0,1]}$, defined for any $T>0$, by \begin{equation*}
     P^T_u = \frac{1}{T}\{ N^T_{Tu} - \int_0^u (\boldsymbol{Id} - \boldsymbol{K}(x))^{-1} \dif x \}_{u \in [0,1]}.
 \end{equation*}
 For any $T>0$ and any $u \in [0,1]$, $P^T_u$ decomposes into three terms
\begin{align*}
    P^T_u
    =
    \frac{1}{T}M^T_{Tu}
    +
    \Big( 
    \frac{1}{T}
    \int_0^{Tu}
        \lambda^T_s
        -
        \mathbb{E}[\lambda^T_s] 
    \dif s
    \Big)
    +
    \frac{1}{T} 
    \Big( 
    \int_0^{Tu} 
        \mathbb{E}[\lambda^T_s]
    \dif s 
    -
    \int_0^u (\boldsymbol{Id} - K(x) )^{-1} \mu(x) \dif x
    \Big),
\end{align*}

say, 
\begin{equation}\label{equ:three_terms_in_Pt}
    P^T_u = \frac{1}{T}
    M^T_{Tu}
    +
    \frac{1}{T}X^T_{Tu}
    +
    Y^T_{Tu},
\end{equation}

 where the martingales $(M^T_{Tu})_{u \in [0,1]}$ were introduced in~\eqref{equ:fundamental_martingale}, and thanks to Lemma~\ref{Lemma:Jaisson_like}, one has for any $T>0$ and $t \in [0,T]$
 \begin{equation*}
     X_t^T
     =
     \int_0^t
     \int_0^s K^T(s,u) \dif M^T_u \dif s.
 \end{equation*}

\subsection{Proof of Theorem~\ref{Th:LLN}}
 We intend to show separately the convergence of each of the three terms $\frac{1}{T}M^T_t$ $,Y^T_t$, and $\frac{1}{T}X^T_t$ from ~\eqref{equ:three_terms_in_Pt}, in Lemmata~\ref{Lemma:martingale_to_0},~\ref{Lemma:deterministic_part_to_0} and~\ref{Lemma:Volterra_part_to_0} respectively. 
\begin{Lemma}\label{Lemma:martingale_to_0}
    Under Assumptions~\ref{ass:g_and_mu_are_C0} to~\ref{ass:stability},
    $\frac{1}{T}
        \sup_{u \in [0,1]}
        \lVert 
        M^T_{Tu}
        \rVert 
        \to
        0
    $,
    as $T \to \infty$, $\Prob-a.s$ and in $L^2(\Prob)$.
\end{Lemma}
\begin{proof}[Proof of Lemma~\ref{Lemma:martingale_to_0}]
    For any $T>0$ and $t \in [0,T]$, by Corollary~\ref{coro:intensity_bound} and Doob's maximal inequality,
    \begin{align*}
     \mathbb{E}\big[ \lVert \frac{1}{T}\sup_{u \in [0,1]} M^T_{Tu} \rVert_2^2 \big]
     \leq 
        \mathbb{E}\big[ \lVert \frac{1}{T} M^T_T \rVert_2^2 \big]
        &\leq
        \frac{1}{T^2}
        \int_0^T
            \mathbb{E}  \big[ \lVert \lambda_t^T \rVert_1 \big]
        \dif t 
        \\
        &\leq \frac{1}{T}\sup_{t \in [0,T]} \mathbb{E}  \big[ \lVert \lambda_t^T \rVert_1 \big] = \mathcal{O}\big( \frac{1}{T} \big)
    .
    \end{align*}
      The almost-sure convergence deduces from $( \frac{1}{T} M^T_t)_{t \in [0,T]}$ being a $L^2(\Prob)$-bounded martingale. 
\end{proof}
\begin{Lemma}\label{Lemma:deterministic_part_to_0}
    Under Assumptions~\ref{ass:g_and_mu_are_C0} to~\ref{ass:stability}, $\sup_{u \in [0,1]} \lVert Y_{Tu}^T \rVert \to 0$ as $T \to \infty$.
\end{Lemma}
\begin{proof}[Proof of Lemma~\ref{Lemma:deterministic_part_to_0}]
    Lemma~\ref{Lemma:deterministic_part_to_0} is a direct consequence of Proposition~\ref{prop:deterministic_convergence}.
\end{proof}

\begin{Lemma}\label{Lemma:Volterra_part_to_0}
    Under Assumptions~\ref{ass:g_and_mu_are_C0} to~\ref{ass:stability},
    $ \frac{1}{T}
    \sup_{u \in [0,1]}
        \lVert X_{Tu}^T \rVert \to 0
    $
    as $T \to \infty$, in $L^2(\Prob)$ and $\Prob$-a.s.
\end{Lemma}
\begin{proof}[Proof of Lemma~\ref{Lemma:Volterra_part_to_0}]
    By Lemma~\ref{Lemma:Jaisson_like} and Corollary~\ref{coro:mean_intensity}, for any $T>0$ and $u \in [0,1]$,
    \begin{equation*}
        X^T_{Tu}
        =
        \frac{1}{T}
        \int_0^{Tu}
        \int_0^t K^T(t,s) \dif M^T_s \dif t
        =
        \frac{1}{T}
        \int_0^{Tu}
            \Big( 
            \int_s^{Tu} K^T(t,s) \dif t
            \Big) 
        \dif M^T_s,
    \end{equation*}
    where we have used Fubini's Theorem. Making the change of variable $t \to t+s$, and, integrating by parts the resulting expression, with $\partial_s K^T(t+s,s) = (\partial_x+\partial_y) K^T(t+s,s)$,
    \begin{equation*}
        X^T_{Tu}
        =
        \int_0^{Tu} \Big(  K^T(Tu,s)  - \int_s^{Tu} (\partial_s + \partial_t)  K^T(t,s) \dif t 
        \Big) \frac{1}{T} M^T_{s} \dif s,
    \end{equation*}
    where the derivative $(\partial_s + \partial_t)  K^T(t,s)$ is well-defined on account of Lemma~\ref{Lemma:IPP_base}. Furthermore, for any $0 \leq s \leq t \leq T$,  one has $\rvert K^T(t,s) \lvert \leq \lVert \bar{\Psi}(t-s) \rVert_1 $ where $\bar{\Psi} \in L^1[0,\infty)$.  As per Lemma~\ref{Lemma:IPP_base} still,  one also has $ \lVert (\partial_s + \partial_t) K^T(t,s) \rVert \leq \frac{1}{T}\lVert \tilde{\Psi} (t-s)\rVert_1$, where $\tilde{\Psi} \in L^1[0,\infty)$. Whence,
    \begin{equation*}
        \lVert X^T_{Tu} \rVert_1 
        \leq 
        \Big(
            \int_0^{\infty}
                \lVert \bar{\Psi}(s)+\tilde{\Psi}(s) \rVert_1 
            \dif s
        \Big)
        \Big( \sup_{s \in [0,T]} \frac{1}{T} \lVert M^T_s \rVert_1
        \Big) ,
    \end{equation*}
    and Lemma~\ref{Lemma:martingale_to_0} allows one to conclude.
\end{proof}

\subsection{Proof of Theorem~\ref{theorem:FTCL}}
Gathering Lemmata~\ref{Lemma:martingale_to_0} to~\ref{Lemma:Volterra_part_to_0}, we have established the functional law of large number of Theorem~\ref{Th:LLN}. We proceed to our central limit Theorem. Introduce the rescaled process 
$
    Z^T_t
    =
    \frac{1}{\sqrt{T}}
    \{
    N^T_t
    -
    \mathbb{E}[N^T_t]
    \}
$.
From the proof of Theorem~\ref{Th:LLN},  for any $T>0$ and $t \in [0,T]$,
\begin{align*}
    Z^T_t
    =
    \frac{1}{\sqrt{T}}
    \{
    M^T_t
    +
    X^T_t
    \}
    =
    \frac{1}{\sqrt{T}}
    M^T_t
    +
    \frac{1}{\sqrt{T}} 
    \int_0^{Tu} 
        \Big( 
        \int_s^{Tu} K^T(t,s) \dif t 
        \Big) 
    \dif M^T_s.
\end{align*}
The difficulty resides in the second term. The rest of this section is thus dedicated to approximating the Volterra process $( T^{- \nicefrac{1}{2}} X^T_t)$ by a more convenient stochastic integral against the martingale measure $\dif M^T_t$. Hereafter we find convenient to define   
\begin{equation}\label{def:gamma}
    \Gamma \colon x \in [0,1]  \mapsto 
    \int_0^{\infty} \Psi(x,s) \dif s \in \mathcal{M}_p(\R).
\end{equation}

\begin{Lemma}[A consequence of Lemma~\ref{lemma:deterministic_approximation}]\label{Lemma:firts_stochastic_approximation}
For any $u \in [0,1]$
    \begin{equation*}
        \frac{1}{\sqrt{T}}
        X^T_{Tu}
        -
        \sum_{i=0}^{\nicefrac{T}{\Delta_T}-1}
        \Big( 
        \int_0^{i \Delta_T u}
            \Psi(iu \frac{\Delta_T}{T},t-s) \dif s
        \Big)
            \frac{1}{\sqrt{T}}
            \big\{ 
            M_{i \Delta_T u + \Delta _T u - s }
            -
            M_{i \Delta_T u   - s }
            \big\}
        \to
        0
    \end{equation*}
    as $T \to \infty$ in $L^2(\Prob)$.
\end{Lemma}
\begin{proof}
Define for any $T>0$ and $u \in [0,1]$
\begin{equation*}
    V^T_{Tu}
    =
    \int_0^{Tu}
    \Big(
    \sum_{i=0}^{ \nicefrac{T}{\Delta _T} -1}
    \int_{i \Delta_T u}^{i \Delta_T u +  \Delta_T u }
        \Psi(i u \frac{\Delta_T }{T},t-s) \dif t
    \Big) 
    \dif M^T_s.
\end{equation*}
Recall that for any $x \in [0,1]$, $\Psi(x,t-s)=0$ when $t<s$, hence the sum in the first integrand of $V^T_t$ effectively runs over the $i \in \mathbb{N}$ such that $(i+1) \Delta_T u > s$. Then, for any $T>0$ and $u \in [0,1]$,
    \begin{align*}
        \frac{1}{\sqrt{T}}
        \big\{ 
            X^T_{Tu}
            -
            V^T_{Tu}
        \big\}
        =
        \int_0^{Tu}
        \Big(
            \int_s^{Tu}
                K^T(t,s)
            \dif t
            -
            \sum_{i=0}^{ \nicefrac{T}{ \Delta_T} -1}
            \int_{i \Delta_T u }^{(i+1)  \Delta_T u}
                \Psi(i u \frac{ \Delta_T}{T}, t-s)
            \dif t 
        \Big) 
        \frac{1}{\sqrt{T}}
        \dif M^T_{s}  ,
    \end{align*}
    and $\mathbb{E}[ 
        \lVert 
        T^{-\frac{1}{2}}
        \{ 
            X^T_{Tu}
            -
            V^T_{Tu}
        \}
        \rVert^2_2
    ]$ is thus bounded by
    \begin{equation*}
    \frac{1}{T}
    \int_0^{Tu}
    \Big\lVert
        \int_s^{Tu}
            K^T(t,s)
        \dif t
        -
        \sum_{i=0}^{ \nicefrac{T}{ \Delta_T} -1}
        \int_{i \Delta_T u }^{(i+1)  \Delta_T u}
            \Psi(i u \frac{ \Delta_T}{T}, t-s)
        \dif t 
    \Big\rVert^2_2
    \mathbb{E}\big[ \lVert \lambda^T_s \rVert_1 \big]
    \dif s.
    \end{equation*}
    
    From Lemma~\ref{lemma:deterministic_approximation} and Corollary~\ref{coro:intensity_bound} we then have some $C>0$ such that
    \begin{equation*}
        \mathbb{E}\big[ 
        \lVert 
        \frac{1}{\sqrt{T}}
        \{ 
            X^T_{Tu}
            -
            V^T_{Tu}
        \}
        \rVert^2_2
    \big]
    \leq 
    C \omega (g,\frac{\Delta_T}{T})^2,
    \end{equation*}
        
    where $\omega(g,\frac{\Delta_T}{T}) \to 0$ as $T \to \infty$. Hence we work with $V_{Tu}^T$, which re-writes as
\begin{align*}
    \sum_{i=0}^{ \nicefrac{T}{ \Delta_T} -1}
    \int_0^{Tu}
        \int_{s}^{(i+1) \Delta_T u }
            \Psi( i u \frac{ \Delta_T}{T}, t-s) 
        \dif t
    \dif M^T_s
    -
    \sum_{i=0}^{ \nicefrac{T}{ \Delta_T} -1}
    \int_0^{Tu}
        \int_{s}^{i  \Delta_T u }
            \Psi( i  u \frac{ \Delta_T}{T}, t-s) 
        \dif t
    \dif M^T_s.
\end{align*}
Integrating by parts, and recalling still that $\Psi(x,t-s)=0$ for any $x \in [0,1]$ and $t < s$, 
\begin{align*}
    V^T_{Tu}
    &= \sum_{i=0}^{ \nicefrac{T}{ \Delta_T} -1}
    \int_0^{(i+1) \Delta_T u    }
        \Psi(i   u \frac{ \Delta_T}{T},s) M^T_{ i  \Delta_T u  +  \Delta_T u  -s} \dif s
    \\
    &-
    \sum_{i=0}^{ \nicefrac{T}{ \Delta_T} -1}
    \int_0^{i  \Delta_T u }
        \Psi(i u \frac{ \Delta_T}{T},s) M^T_{ i u  \Delta_T -s} \dif s.
\end{align*}
Re-arranging the terms,
\begin{align*}
    \frac{1}{\sqrt{T}} V^T_{Tu}
    &=
    \frac{1}{\sqrt{T}}
    \sum_{i=0}^{ \nicefrac{T}{ \Delta_T} -1}
    \int_0^{i  \Delta_T u}
        \Psi(k u \frac{ \Delta_T}{T},s)
        \big\{
            M^T_{ i  \Delta_T u   +  \Delta_T u-s} 
            -
            M^T_{i  \Delta_T u -s}
        \big\} 
    \dif s
    \\
    &+
    \frac{1}{\sqrt{T}}
    \sum_{i=0}^{ \nicefrac{T}{ \Delta_T} -1}
    \int_{i  \Delta_T u}^{(i+1) \Delta_T u} \Psi(i u \frac{ \Delta_T}{T},s) M^T_{i \Delta_T u +  \Delta_T u-s} \dif s.
\end{align*}

The first term of the preceding sum corresponds to the desired estimate.  Recalling that $\lVert 
 \Psi(x,s) \rVert_1$ is less than $\lVert \Bar{\Psi}(s) \rVert_1$, the second one is bounded by $\frac{1}{\sqrt{T}}
    \sup_{s < u \Delta_T} \lVert M^T_{s} \rVert_1 (\int_{0}^{\infty} \lVert \Bar{\Psi} ( t) \rVert_1 \dif t )$ and thus goes to  $0$ as $T \to \infty$ by Doob's maximal inequality, whereupon the whole expression vanishes as $T \to \infty$, ending the proof. 
\end{proof}

\begin{Lemma}\label{Lemma:second_stochastic_approximation}
Under Assumption~\ref{ass:stability}, for any $u \in [0,1]$,
\begin{equation*}
    \frac{1}{\sqrt{T}}
    X^T_{Tu}
    -
    \sum_{i=0}^{ \nicefrac{T}{\Delta_T} -1}
    \Big( 
    \int_0^{\infty}
        \psi( i u \frac{\Delta_T}{T}, s) 
    \dif s
    \Big) 
    \frac{1}{\sqrt{T}}
    \big\{ 
        M_{i \Delta_T u +\Delta_Tu}
        -
        M_{i \Delta_T u}
    \big\}
    \to 0
\end{equation*}
as $T \to \infty$ in $L^1(\Prob)$.
\end{Lemma}
\begin{proof}
Define at any $u  \in [0,1]$
\begin{align}
U^T_{u}&=
    \sum_{i=0}^{\nicefrac{T}{\Delta_T} -1}
    \int_0^{iu\Delta_T}
    \Psi( i u \frac{\Delta_T}{T} , s) 
    \frac{1}{\sqrt{T}}
    \big\{ 
        M_{i\Delta_T u +\Delta_Tu -s}-M_{i \Delta_T u-s}
    \big\}
    \dif s
    \label{equ:U_T_to0}
    \\
    &-
    \sum_{i=0}^{\nicefrac{T}{\Delta_T} -1}
    \int_0^{\infty}
        \Psi( i u \frac{\Delta_T}{T}, s) 
    \dif s
    \frac{1}{\sqrt{T}}
    \big\{ 
        M_{i \Delta_Tu +\Delta_T u}-M_{i \Delta_T u}
    \big\}.
    \nonumber
\end{align}

In light of Lemma~\ref{Lemma:firts_stochastic_approximation}, it suffices to show that $U^T_{u}$ converges to $0$ as $T \to \infty$. To this effect, we adapt an argument introduced in Bacry \textit{et al.}~\cite[proof of Theorem 3]{bacrylimit}. Let $u \in [0,1]$. Let us first remark that, by Corollary~\ref{coro:intensity_bound}, one has some $C>0$ such that  $\mathbb{E}[ \lVert M^T_{ \Delta_T u} \rVert^2_2] \leq \mathbb{E}[ \int^{\Delta_T}_0  \lVert \lambda_t\rVert_1 \dif t]  \leq C \Delta_T$. Hence the two terms at $i=0$ in~\eqref{equ:U_T_to0} are respectively null and in $\mathcal{O}(\sqrt{\nicefrac{\Delta_T}{T}})$ as $T \to \infty$, and we work only with these wherein $i \geq 1$.  Consider now the first integral on the right-hand side of~\eqref{equ:U_T_to0}. Working once more \textit{à la} Césàro, we let $\eta \in [0,u]$, and detach the integrands within which $s>\eta \Delta_T $. By Fubini's Theorem, for any $T >0$,

\begin{align*}
\sum_{i=1}^{ \nicefrac{T}{\Delta_T} -1}
    \int_{\eta \Delta_T }^{i\Delta_T u}
    \Psi( i \frac{\Delta_T}{T} &u, s) 
    \frac{1}{\sqrt{T}}\big\{ 
        M_{i\Delta_T u +\Delta_Tu -s}-M_{i \Delta_T u-s}
    \big\}
    \dif s
    \\
    &=
    \int_0^{\infty}
    \sum_{i=1}^{\lfloor \nicefrac{T}{\Delta_T} \rfloor}
    \mathbb{1}_{\{ \eta \Delta_T  \leq s \leq i \Delta_T u \}}
    \Psi( i \frac{\Delta_T}{T} u, s) 
    \frac{1}{\sqrt{T}}
    \big\{
        M_{i\Delta_T u +\Delta_Tu -s}-M_{i \Delta_T u-s}
    \big\}
    \dif s
    \\
    &=
    \int_0^{\infty} 
        S_{\lfloor \nicefrac{T}{\Delta_T} \rfloor}(s,u,T)
    \dif s,
\end{align*}
where, for any $s \in [0,T]$, the process $(S_n(s,u,T))_{n \geq 1}$ defined by
\begin{equation*}
    S_n(s,u,T)=
    \sum_{i=1}^n 
    \mathbb{1}_{\{ \Delta_T \eta \leq s \leq i \Delta_T u \}}
    \Psi( i \frac{\Delta_T}{T} u, s) 
    \frac{1}{\sqrt{T}}
    \big\{  
        M_{i\Delta_T u +\Delta_Tu -s}-M_{i \Delta_T u-s}
    \big\}
\end{equation*}
is a discrete martingale for the filtration $(\mathcal{F}_{(i+1)\Delta_T u -s})_i$. Using Corollary~\ref{coro:intensity_bound}, one has some $C>0$ such that $\mathbb{E}[ \lVert M_{(i+1)\Delta_T u -s}- M_{i \Delta_T u -s} \rVert_2^2 ] \leq C \Delta_T u$, hence for any $n \in \mathbb{N}$
\begin{equation*}
    \mathbb{E}\big[ \lVert S_n(s,u,T) \rVert_2^2 \big]
    \leq 
    n
    C
      \lVert  \Bar{\Psi}(s) 
     \rVert_2^2
     \frac{\Delta_T}{T}
      \mathbb{1}_{\{ \eta \Delta_T \leq s \leq i \Delta_T u \}},
\end{equation*}
where we have also used the fact that $\Psi(x,s) \leq \Bar{\Psi}(s)$ for any $x \in [0,1]$. At $n = \lfloor\nicefrac{T}{\Delta_T} \rfloor$ in particular, by Jensen's inequality applied to $x \mapsto \sqrt{x}$, one has some $C'>0$  such that
\begin{align*}
\mathbb{E}\Big[ 
    \big\lVert
    \sum_{i=1}^{ \nicefrac{T}{\Delta_T} -1 }
    \int_{\eta \Delta_T }^{i\Delta_T u}
    \Psi( i \frac{\Delta_T}{T} u, s) 
    &\frac{1}{\sqrt{T}}\big\{ 
        M_{i\Delta_T u +\Delta_Tu -s}
        -
        M_{i \Delta_T u-s}
    \big\}
    \dif s
    \big\rVert_1
    \Big]
    \\
    &\leq 
    \int_0^{\infty} 
    \mathbb{E}\big[ \lVert S_{\lfloor\nicefrac{T}{\Delta_T \rfloor}}(s,u,T) \rVert_1\big] \dif s
    \leq
    C'\int_{\eta \Delta_T}^{\infty} \lVert \Bar{\Psi} (s) \rVert_1 \dif s,
\end{align*}
which goes to $0$ as $T \to \infty$  in view of the integrability of $\Bar{\Psi}$. The remaining terms in~\eqref{equ:U_T_to0} express as 
    \begin{equation*}
        \int_0^{\infty}
        \sum_{i=1}^{ \nicefrac{T}{\Delta_T} -1 }
        \Psi(i u\frac{\Delta_T}{T},s) \Big(
        \mathbb{1}_{s< \eta \Delta_T }
    \frac{1}{\sqrt{T}}
    \big\{ 
        M_{i \Delta_Tu + \Delta_T u -s}-M_{i  \Delta_T u -s}
    \big\}
    -
    \frac{1}{\sqrt{T}}
    \big\{ 
        M_{i  \Delta_T u  +  \Delta_T}-M_{i  \Delta_T}
    \big\}
    \Big)
    \dif s.
    \end{equation*}

Proceeding in the same fashion as above, we introduce the discrete martingale
\begin{equation*}
    \tilde{S}_n(u,s,T)=\sum_{i=1}^n \Psi(i u \frac{\Delta_T}{T},s)
    \frac{1}{\sqrt{T}}        \Big(  \mathbb{1}_{ \{s<\eta  \Delta_T  \} }\big\{ 
        M_{i \Delta_T u +\Delta_Tu-s}-M_{i \Delta_Tu-s}
    \big\}
    -
    \big\{ 
        M_{i \Delta_T u +\Delta_T u}-M_{i \Delta_T u}
    \big\} \Big).
\end{equation*}
 For any $u \in [0,1]$ and any $s>0$, $(\tilde{S}_n(u,s,T))_n$ is adapted to $(\tilde{\mathcal{F}}^{u,T}_i)\vcentcolon=(\mathcal{F}_{(i+1)\Delta_T u })$. Using again the orthogonality of martingale increments, $\mathbb{E}[\lVert \tilde{S}_n(u,s,T)\rVert_2^2$ is bounded by
\begin{equation*}
    \sum_{i=1}^n
    \lVert \Psi (i u \frac{\Delta_T}{T},s) \rVert_2^2
    \mathbb{E}\Big[
    \big\lVert
     \mathbb{1}_{s < \eta \Delta_T } 
    \big\{  
        M_{i \Delta_T u+\Delta_T u-s}-M_{i \Delta_T u-s}
    \big\}
    -
    \big(\{
        M_{i \Delta_T u+\Delta_T u }-M_{i\Delta_T u}
    \big\}
    \big\rVert^2_2
    \Big].
\end{equation*}
Recall that we have chosen $\eta < u$. An orthogonality argument then shows that the expectations in the summands hereinabove are respectively bounded by
\begin{align*}
    \int_{i \Delta_T u +\Delta_T u -s}^{i \Delta_T u+\Delta_T u}
        \mathbb{E} \big[ \lVert \lambda^T_r \rVert_1 ] 
    \dif r
    +
    \int_{i \Delta_T u -s}^{i \Delta_T u}
       \mathbb{E} \big[
        \lVert \lambda^T_{r} \rVert_1
       \big]
    \dif r,
\end{align*}
if $s \leq  \eta \Delta_T$, and by$
    \int_{i u \Delta_T}^{(i+1)u \Delta_T} \mathbb{E}[ \lVert \lambda^T_r \rVert_1 ] \dif r
$ if $s \geq \eta \Delta_T$.  Again, by Corollary~\ref{coro:intensity_bound}, $\mathbb{E}[ \lVert \lambda^T_t \rVert_1 ]$ is bounded uniformly in $(t,T)$. Altogether, one has some $C>0$ such that $ \mathbb{E}\big[\lVert \tilde{S}_n(u,s,T)\rVert_2^2\big]$ is less than
\begin{equation}\label{equ:majoration}
    C
    \Big( 
    \sum_{i=1}^{n}
    \mathbb{1}_{\{ s \leq \eta \Delta_T\}}
    \lVert \Psi (i u \frac{\Delta_T}{T},s) \rVert_2^2
    \eta \Delta_T
    +
    \sum_{i=1}^{n}
    \mathbb{1}_{\{ s > \eta \Delta_T\}}
    \lVert \Psi (i u \frac{\Delta_T}{T},s) \rVert_2^2
    u \Delta_T
    \Big).
\end{equation}
Finally, Jensen's inequality applied to $x \mapsto \sqrt{x}$ together with ~\eqref{equ:majoration} shows that
\begin{align*}
    \int_0^{\infty}
    \mathbb{E}\Big[ 
    \Big\lVert 
    \sum_{i=1}^{ \nicefrac{T}{\Delta_T} -1 }
        \Psi(i u \frac{\Delta_T}{T},s)
        \mathbb{1}_{s< \eta \Delta_T}\big\{
        M^T_{(i+1)u\Delta_T-s}
        -
        M^T_{iu\Delta_T-s}
        \}
        &-
        \big\{
        M^T_{(i+1)u\Delta_T}
        -
        M^T_{iu\Delta_T}
        \}
    \big)
    \Big\rVert_2 
    \Big]
    \dif s
\end{align*}
is dominated, up to some multiplicative constant, by
\begin{equation*}
    \Big( \int_0^{\infty}
    \lVert \bar{\Psi}(s) \rVert_1 \dif s \Big)
    \sqrt{\eta}
    +
    \Big( \int_{ \eta \Delta_T }^{\infty}
    \lVert \bar{\Psi}(s) \rVert_1 \dif s \Big).
\end{equation*}
By integrability of $\bar{\Psi}$, the last term hereinabove tends to $0$ as $T \to \infty$. As $\eta$ is freely taken in $[0,u]$ this ends the proof. 
\end{proof}

\begin{Lemma}\label{Lemma:next_to_last_stochastic_approximation}
Under Assumption~\ref{ass:stability}, with $\Gamma$ defined in~\eqref{def:gamma}, for any $u \in [0,1]$,
\begin{equation*} 
   \sum_{i=0}^{ \nicefrac{T}{\Delta_T} -1}
   \Big( 
   \int_0^{\infty}
        \Psi( i \frac{\Delta_T}{T} u,s)
    \dif s
    \Big)
    \frac{1}{\sqrt{T}}\big\{ 
        M_{i \Delta_T u + \Delta_T u}
        -
        M_{i \Delta_T u }
    \big\}
    -
    \int_0^{u}
    \Gamma(s) \frac{1}{\sqrt{T}}
    \dif M^T_{sT}
    \xrightarrow[T \to \infty]{L^2(\Prob)}
    0
\end{equation*}
\end{Lemma}
\begin{proof}
    Let $T>0$ and define for any $u \in [0,1]$
    \begin{equation*}
        \mathcal{Z}^T_u=
        \sum_{i=0}^{ \nicefrac{T}{\Delta_T} -1} \int_0^{\infty} \Psi( i \frac{\Delta_T}{T} u,s) \dif s 
        \frac{1}{\sqrt{T}}
    \big\{
    M_{i \Delta_T u + \Delta_T u}
    -
    M_{i \Delta_T u }
    \big\}
    -
    \int_0^{u}
    \Gamma(s)
    \frac{1}{\sqrt{T}} \dif  M^T_{sT}.
    \end{equation*}
    The process $(\mathcal{Z}_{u}^T)$ re-writes as the stochastic integral
    \begin{equation*}
        \mathcal{Z}^T_{u}
        =
        \int_0^u
        \Big( 
       \sum_{i=0}^{ \nicefrac{T}{\Delta_T} -1}
        \mathbb{1}_{\{ i\Delta_T u \leq s T \leq (i+1)\Delta_T u\}} 
        \big( 
        \Gamma(i u \frac{\Delta_T}{T}) - \Gamma(s)
       \big) \Big) \dif \Tilde{M}^T_s, 
    \end{equation*}
    where $(\tilde{M}_u)_{u \in [0,1]}$ is the rescaled martingale $(T^{-\nicefrac{1}{2}}M_{uT}^T)_{u \in [0,1]}$. Hence for any $u \in [0,1]$,
    \begin{equation}\label{equ:W_control}
        \mathbb{E}\big[ 
            \lVert \mathcal{Z}^T_{u} \rVert^2_2
        \big]
        \leq 
        C
        \sum_{i=0}^{\lfloor \nicefrac{T}{\Delta_T} \rfloor}
        \int_{i \frac{\Delta_T}{T} u}
        ^{i \frac{\Delta_T}{T}u + \frac{\Delta_T}{T} u}
        \lVert 
            \Gamma(i u \frac{\Delta_T}{T}) 
            -
            \Gamma(s)
        \rVert^2_2
        \dif s,
    \end{equation}
    where $ C= 
        \sup_{t \in [0,T], T>0} 
        \mathbb{E}[ \lVert \lambda^T_t \rVert_1 ]
        $ is indeed finite and bounded as per Corollary~\ref{coro:intensity_bound}. Moreover, we have established in Lemma~\ref{Lemma:gamma_continuity} the uniform continuity of $\Gamma$ over $[0,1]$.  Letting $\varepsilon>0$, the integrands in~\eqref{equ:W_control} then  fall uniformly below $\nicefrac{\varepsilon}{2C}$ as the mesh size $\nicefrac{\Delta_T}{T}$ tends to $0$. Whence the bound itself falls below $ C\sum_{i=0}^{\lfloor \nicefrac{T}{\Delta_T} \rfloor} \frac{\Delta_T}{T} \frac{1}{2C}\varepsilon \leq  \varepsilon$ as $T \to \infty$.
\end{proof}

Pulling Lemmata~\ref{Lemma:firts_stochastic_approximation},~\ref{Lemma:second_stochastic_approximation} and~\ref{Lemma:next_to_last_stochastic_approximation} together, we have established that for any $u \in [0,1]$,
\begin{equation}\label{equ:X_is_dM_integral}
    \frac{1}{\sqrt{T}} 
    \big\{ N_{Tu} - \mathbb{E}\big[ N_{Tu} \big]\big\}
    -
    \big\{ 
        \frac{1}{\sqrt{T}}M^{T}_{Tu}
        +
        \int_0^u \Gamma(x) \frac{1}{\sqrt{T}}\dif  M^T_{Tx}
    \big\}
    \xrightarrow[T \to \infty]{\Prob} 0.
\end{equation}

\subsection{Completion of the proof of Theorem~\ref{theorem:FTCL}} Recall that $
    \boldsymbol{Id}+\Gamma(x) = (\boldsymbol{Id}-\boldsymbol{K}(x))^{-1}
$ for any $x \in [0,1]$, hence~\eqref{equ:X_is_dM_integral} is equivalent to
\begin{equation}\label{equ:X_is_dM_integral_bis}
    \frac{1}{\sqrt{T}} 
    \big\{ N_{Tu} - \mathbb{E}\big[ N_{Tu} \big]\big\}
    -
        \int_0^u 
            (\boldsymbol{Id}-\boldsymbol{K}(x))^{-1} \frac{1}{\sqrt{T}}
        \dif  M^T_{Tx}
    \xrightarrow[T \to \infty]{\Prob} 0.
\end{equation}

The convergence will hold uniformly in $u$ provided we can show the tightness of the involved family of processes. For now, let us prove the weak convergence of the right-hand term towards the desired limit.

\begin{Lemma}\label{Lemma:skoro_conv}
    Under Assumptions~\ref{ass:g_and_mu_are_C0} and~\ref{ass:stability},
    \begin{equation*}
         \int_0^u 
            (\boldsymbol{Id}-\boldsymbol{K}(x))^{-1} \frac{1}{\sqrt{T}}
        \dif  M^T_{Tx}
        \xrightarrow[T \to \infty]{\mathcal{L}(\Prob)}
        \int_0^u \mathfrak{S}(x) \dif W_x
    \end{equation*}
    in Skorokhod topology with $\mathfrak{S}(x) = (\boldsymbol{Id}-\boldsymbol{K}(x))^{-1} \boldsymbol{\Sigma}(x)$ with  $\boldsymbol{\Sigma}$ as in Theorem~\ref{theorem:FTCL}.
\end{Lemma}
\begin{proof}[Proof.]
Let us show the rescaled martingale $(\tilde{M}_u^T) = \frac{1}{\sqrt{T}}(M^T_{Tu})$ converges towards a Gaussian process with covariance $\boldsymbol{\Sigma}$, as defined in Theorem~\ref{theorem:FTCL}. This follows from the fact that
\begin{equation*}
    [ \tilde{M}^T, \tilde{M}^T ]_u 
    =
    \frac{1}{T} \textup{Diag}(N_{Tu}) \to  \boldsymbol{\Sigma}(u)  \textup{  as  } T \to \infty,
\end{equation*}
in $L^2(\Prob)$ as per Theorem~\ref{Th:LLN}, which is a sufficient condition for the martingale central limit Theorem since the jumps of $T^{-\nicefrac{1}{2}} M^T_t$ are uniformly bounded by $T^{-\nicefrac{1}{2}}$. Observe furthermore that, $g$ being continuously differentiable, so is $\Gamma$ and so is $(\boldsymbol{Id}-\boldsymbol{\Gamma})^{-1}$. The function mapping $(X_u) \in D[0,1]$ to $(\int_0^u (\boldsymbol{Id}-\boldsymbol{K}(x))^{-1} \dif X_x) \in D[0,1]$ is then continuous in Skorokhod topology and the Lemma deduces from the continuous mapping Theorem. Note furthermore that the convergence occurs over the compact interval $[0,1]$ and the limit process is continuous, hence the results extends to the topology of uniform convergence.  
\end{proof}

Lemma~\ref{Lemma:skoro_conv} together with~\eqref{Lemma:skoro_conv} yields the convergence of Theorem~\ref{theorem:FTCL} in the sense of finite dimensional distributions. As we work with the null initial condition $N_0=0$, the convergence will also occur in Skorokhod topology provided the process $\frac{1}{\sqrt{T}} 
    \big\{ N_{Tu} - \mathbb{E}\big[ N_{Tu} \big]\big\}$ is tight, as shown in Proposition~\ref{prop:tight} below, before which we require Lemma~\ref{lemma:required_for_tightness}. Corollary~\ref{coro:sqrt_coro} will then follow from the second assertion of Proposition~\ref{prop:deterministic_convergence}.

    \begin{Lemma}\label{lemma:required_for_tightness}
   Under Assumptions~\ref{ass:g_and_mu_are_C0} to~\ref{ass:disease}, the family of functions $(t \mapsto \int_0^t K^T(t,s) \dif s)$ indexed on $T >0$ is uniformly equi-continuous.  
\end{Lemma}

\begin{proof}
let $T>0$. Let also $(t',t) \in [0,T]^2$. Recall that $K^T$ has support in the simplex $\{t,s \in [0,T] \hspace{0.1cm} \lvert \hspace{0.1cm} t>s\}$, hence, 
\begin{equation}\label{equ:to_control}
            \int_0^{t'}  
            K^T(t',s) 
        \dif s
        -
        \int_0^t 
            K^T(t,s)
        \dif s
        =
        \int_0^T K^T(t',s) - K^T(t,s) \dif s.
\end{equation}
Introduce the function $f^T \colon (t,s,x)  \in [0,T]^3  \mapsto    K^T(t+x,s) - K^T(t,s) $. Controlling~\eqref{equ:to_control} reduces to bounding $\int_0^T f^T(t,x,s) \dif s$.  For any $(t,x,s) \in [0,T]^3$,
    \begin{align}
        f^T(t,x,s)
        &=
        \Big(
            g(\frac{t+x}{T})
            -
            g(\frac{t}{T})
        \Big)
        \varphi(t+x-s) 
        \nonumber
        \\
        &+
        g(\frac{t}{T})
        \Big( 
            \varphi(t+x-s)
            -
            \varphi(t-s)
        \Big) 
        \nonumber
        \\
        &+
        \int_s^{t+x} 
            \big( g(\frac{t+x}{T})- g(\frac{t}{T})\big)\varphi(t+x- u) 
            K^T(u,s)
        \dif u
        \label{equ:baseline}
        \\
        &+
        \int_s^{s+x} 
            g(\frac{t}{T}) \varphi(t+x- u) 
            K^T(u,s)
        \dif u
        \nonumber
        \\
        &+
        \int_0^{t-s} 
            g(\frac{t}{T} )\varphi(u) 
            \big(
            K^T(t+x-u,s)
            -
            K^T(t-u,s) 
            \big) 
        \dif u.
        \nonumber
    \end{align}
    Denote by $m^T(t,x,s)$ the sum of the four first terms on the right-hand side of~\eqref{equ:baseline}. Then, using again that $K^T(u,v)=0$ when $u \leq v$  and $K^T(u,v) \geq 0$ otherwise, one recovers  the coordinate-wise renewal inequality
    \begin{equation}\label{equ:some_more_renewal}
        f^T(t,x,s)
        \leq 
        m^T(t,x,s)
        +
        \int_0^t \Bar{\varphi}(u) f^T(t-u,x,s) \dif u, \hspace{0.1cm} t \in [0,T].
    \end{equation}
    Inequation~\eqref{equ:some_more_renewal} resolves into
    \begin{equation*}
        f^T(t,x,s) 
        \leq m^T(t,x,s)
        +
        \int_0^t \Bar{\Psi}(t-u)m^T(u,x,s) \dif u, \hspace{0.1cm} t \in [0,T],
    \end{equation*}
    where $\Bar{\Psi}= \sum_{i \geq 1}^{\infty} \Bar{\varphi}^{\star i}$ is the resolvent of the convolution kernel $t,s \mapsto \Bar{\varphi}(t-s)$. Integrating, using Fubini's Theorem and Young's convolution inequality,
    \begin{align*}
    \big\lVert 
        \int_0^T f^T(t,x,s) \dif s
    \big\rVert 
    &\leq 
    \big\lVert 
        \int_0^T m^T(t,x,s) \dif s
        +
        \int_0^T
        \Big( 
        \int_0^t
        \Bar{\Psi}(t-u)
         m^T(u,x,s)  \dif u  \Big) \dif s
    \big\rVert 
    \\
    &\leq 
    \big\lVert 
        \int_0^T m^T(t,x,s) \dif s
        +
        \int_0^t
        \Bar{\Psi}(t-u)
        \Big( 
        \int_0^T 
            m^T(u,x,s)
        \dif s
        \Big)
        \dif  u
    \big\rVert 
    \\
    &\leq 
    (1+
    \int_0^{\infty} \lVert \Bar{\Psi}(s) \rVert  \dif s)
    \sup_{u \in [0,T]}
    \big\lVert 
        \int_0^T m^T(u,x,s) \dif s 
    \big \rVert.
    \end{align*}
    The Lemma follows from the last multiplicative factor hereinabove tending to $0$ as $x \to 0$, as we now establish. Let $u \in [0,T]$. Firstly,
    \begin{equation*}
        \lVert \int_0^T \Big( g\big(\frac{u+x}{T}\big)-g\big(\frac{u}{T} \big) \Big) \varphi(u+x-s) \dif s
        \rVert \leq \omega(g,x) \lVert \varphi \rVert_{L^1},
    \end{equation*}
    where the modulus of continuity $\omega(g,x)$ of $g$ at $x$ goes to $0$ as $x \to 0$.
    Secondly, 
    \begin{align}\label{equ:holder_from_here}
        g(\frac{u}{T})
        \int_0^T \varphi(u+x-s)-\varphi(u-s) \dif s
        =
        g(\frac{u}{T})
        \int_u^{u+x} \varphi(s) \dif s.
    \end{align}
    By Hölder's inequality we have some $\nu\geq 1$ such that $\nu^{-1} + (1+\epsilon)^{-1} = 1$ and, for any $v>0$,  $\lVert \int_v^{v+x} \varphi(s) \dif s \rVert \leq x^{\frac{1}{\nu}} \lVert \varphi \rVert_{L^{1+\epsilon}[0,\infty)}$, hence
    \begin{align*}
        \lVert 
        g(\frac{u}{T})
        \int_0^T \varphi(u+x-s)-\varphi(u-s) \dif s
        \rVert 
        \leq 
         x^{\frac{1}{\nu}} 
        \lVert g \rVert_{L^{\infty}[0,1]}
        \lVert \varphi \rVert_{L^{1+\epsilon}[0,\infty)}
        \xrightarrow[x \to 0]{}
        0.
    \end{align*}
    Applying again Young's convolution inequality to the third term on the right-hand side of~\eqref{equ:baseline} yields a bound of the form
    \begin{equation*}
        \Big\lVert 
        \int_0^T
        \Big( 
        g\big(\frac{u+t}{T} \big)-g\big(\frac{u}{T}\big)
        \Big)
        \int_s^{u+x}
        \varphi(u+x-v) K(v,s) \dif v \dif s
        \Big\rVert 
        \leq 
        \omega(g,x) \lVert \varphi 
        \rVert_{L^1[0,\infty)}
        \lVert \Bar{\Psi} \rVert_{L^1[0,\infty)},
    \end{equation*}
    which goes to $0$ as $x \to \infty$. Likewise, using Fubini's Theorem with the fourth term,
    \begin{equation*}
       \lVert 
       \int_0^T g(\frac{u}{T}) \int_s^{s+x} \varphi(u+x-v)K(v,s) \dif v \dif s
       \rVert 
       \leq
       \lVert \Bar{\varphi}\rVert_{L^1[0,\infty)}
       \int_{u-x}^u 
            \rVert \Bar{\Psi}(u-v) \rVert
        \dif v,
    \end{equation*}
    where $\int_{u-x}^u 
            \rVert \Bar{\Psi}(u-v) \rVert
        \dif v = \int_0^x  \rVert \Bar{\Psi}(v) \rVert \dif v \to 0$ as $T \to \infty$, ending the proof.

\end{proof}
    
\begin{proposition}\label{prop:tight}
  Work under Assumptions~\ref{ass:g_and_mu_are_C0} to~\ref{ass:disease}. The family of processes 
  \begin{equation*}
  \Big\{
  T^{-\frac{1}{2}}
  ( 
      N^T_{Tu} 
        -
    \mathbf{E}[N_{Tu}]
), \hspace{0.1cm} u \in [0,1]
\Big\}_{T >0}
  \end{equation*}
  is tight in $D[0,1]$ equipped with the topology of uniform convergence.
\end{proposition}
\begin{proof}
    Recall that for any $T>0$ and $u \in [0,1]$,
    \begin{equation}\label{equ:Rosenbaum_decomposition}
        N^T_{Tu} 
        -
        \mathbf{E}[N_{Tu}]
        =
        M_{Tu}^T
        +
        \int_0^{Tu} \int_0^t K^T(t,s) \dif M^T_s \dif t
        =
        M^T_{Tu} + Y^T_{Tu}.
    \end{equation}
    The martingale $T ^{-\nicefrac{1}{2}}( M^T_{Tu})$ is tight as per its convergence in uniform topology towards a $\Prob$-a.s continuous limit. Hence it suffices to show the tightness of $T^{-\nicefrac{1}{2}} (Y^T_{Tu})$, which reads
    \begin{equation*}
    \lim_{\delta \to 0} \limsup_{T \to \infty}
        \Prob\Big[
        \frac{1}{\sqrt{T}}
        \sup_{\lvert u-v\rvert < \delta} \lVert X^T_{Tv} - X^T_{Tu}\rVert
        > \eta \Big]
        =0 \textup{  for any  }
        \eta >0.
    \end{equation*}

    For any $u\in [0,1]$ then, integrating by parts, 
    \begin{align*}
        Y^T_{Tv}
        =
        \int_0^{Tu}
        \Big( 
            K^T(Tu,s)
            -
            \int_s^{Tu}
            (\partial_s
            +
            \partial_t)
            K^T(t,s) 
        \Big)  
        \frac{M^T_s}{\sqrt{T}}
        \dif s.
    \end{align*}
    Hence for any $(u,v) \in [0,1]^2$,
    \begin{align*}
        Y^T_{Tv}
        -
        Y^T_{Tu}
        &=
        \int_{0}^{Tu}
        \Big( 
            \int_{Tu}^{Tv}
            (\partial_s
            +
            \partial_t)
            K^T(t,s) 
            \dif t
        \Big)  
        \frac{M^T_s}{\sqrt{T}}
        \dif s
        \\
        &-
        \int_{Tu}^{Tv}
        \Big( 
            \int_s^{Tv}
            (\partial_s
            +
            \partial_t)
            K^T(t,s) 
            \dif t
        \Big)  
        \frac{M^T_s}{\sqrt{T}}
        \dif s
        \\
        &+
        \int_{0}^{Tv}
        K^T(Tv,s) 
        \frac{M^T_s}{\sqrt{T}}
        \dif s
        -
        \int_{0}^{Tu}
        K^T(Tu,s) 
        \frac{M^T_s}{\sqrt{T}}
        \dif s.
        \\&=\textbf{\textsc{i}}+\textbf{\textsc{ii}}+\textbf{\textsc{iii}}
    \end{align*}

    Let $\delta>0$. Let $(u,v) \in [0,1]^2$ such that $0 \leq v-u \leq \delta$. Firstly,
    \begin{equation*}
        \textbf{\textsc{i}}
        \leq 
        \sup_{s \in [0,T]} \frac{1}{\sqrt{T}}
        \lVert M^T_s \rVert 
        \int_{Tu}^{Tv}
        T^{-1} \Big( \int_0^{\infty} \lVert \Tilde{\Psi} (s)\rVert \dif s \Big) \dif t
        = \mathcal{O}_{\Prob}(\delta ).
    \end{equation*}

    Secondly, for any $T >0$,
\begin{equation*}
    \textbf{\textsc{ii}}
    \leq 
    \sup_{s \in [Tu,Tv]} 
    \frac{1}{\sqrt{T}} 
    \lVert M^T_s \rVert 
    \frac{1}{\sqrt{T}}
    \int_{Tu}^{\infty} \lVert \Tilde{\Psi} (s) \rVert
    \dif s.
\end{equation*}
Let $\eta \in [0,1)$. When $v \in [0,\eta]$, 
\begin{equation*}
    \textbf{\textsc{ii}}
    \leq 
    \Big( \int_0^{\infty} \lVert \Bar{\Psi}(s) \rVert_1 \dif s
    \Big)
    \frac{1}{\sqrt{T}}
    \sup_{s \in [0,\eta T]}
    \lVert M^T_s \rVert_1 = \mathcal{O}(\eta),
\end{equation*}
whereas if $v \in (\eta,1]$,
\begin{equation*}
    \textbf{\textsc{ii}}
    \leq 
    \Big( \int_{Tu}^{\infty} \lVert \Bar{\Psi}(s) \rVert_1 \dif s
    \Big)
    \frac{1}{\sqrt{T}}
    \sup_{s \in [0,T]}
    \lVert M^T_s \rVert_1 = \mathcal{O}\big(\int^{\infty}_{(\eta-\delta) T} \lVert \Bar{\Psi}(s) \rVert \dif s \big).
\end{equation*}
The second term $\textbf{\textsc{ii}}$ is thus arbitrarily small as $T \to \infty$. We may in fact proceed in a similar fashion for the third and remaining term. Let $\eta \in (0,1]$. On one hand,
\begin{align*}
     \lVert\int_{T( v \land \eta)}^{Tv} 
    K^T(Tv,Tv-s) \frac{M^T_{Tv-s}}{\sqrt{T}}
    \dif s
    \rVert 
    &\leq 
    \lVert
    \int_{T( v \land \eta)}^{Tv} K^T(Tv,Tv-s) \dif s
    \rVert
    \sup_{t \in[0,T]}
    \lVert \frac{M^T_{Tv-s}}{\sqrt{T}}
    \rVert 
    \\
    &\leq 
    \int_{T( v \land \eta)}^{Tv} 
    \lVert \Bar{\Psi}(s) \rVert  
    \dif s
    \sup_{t \in[0,T]}
    \lVert \frac{M^T_{t}}{\sqrt{T}}
    \rVert
    \\
    &\leq
    \Big( 
        \int_{T \eta}^{\infty} 
            \lVert \Bar{\Psi}(s) \rVert  
         \dif s 
    \Big)
    \sup_{t \in[0,T]}
    \lVert \frac{M^T_{t}}{\sqrt{T}}
    \rVert,
\end{align*}
and likewise,
\begin{align*}
     \lVert\int_{T( u \land \eta)}^{Tu} 
    K^T(Tu,Tu-s) \frac{M^T_{Tu-s}}{\sqrt{T}}
    \dif s
    \rVert 
    \leq 
    \Big(  
        \int_{T (\eta-\delta)}^{\infty} 
            \lVert \Bar{\Psi}(s) \rVert  
        \dif s
    \Big) 
    \sup_{t \in[0,T]}
    \lVert \frac{M^T_{t}}{\sqrt{T}}
    \rVert,
\end{align*}

where $ \sup_{t \in[0,T]}
    \lVert T^{-\nicefrac{1}{2}}M^T_{t} \rVert = \mathcal{O}_{\Prob}(1) $ and the two bounds hereinabove vanish as $T \to \infty$. On the other,
    \begin{align*}
        \int_0^{T(u\land \eta)} K^T(Tu,Tu-s) \frac{M^T_{Tu-s} }{\sqrt{T}} \dif s
        &-
        \Big( \int_0^{T(u\land \eta)} K^T(Tu,Tu-s) \dif s \Big) \frac{M^T_{Tu} }{\sqrt{T}} 
        \\
        &\leq 
        \Big (\int_0^{\infty} \lVert \Bar{\Psi}(s) \rVert \dif s\Big)
        \sup_{ \lvert x - y \rvert \leq \eta}
        \Big\lVert \frac{M^T_{Tx}-M^T_{Ty}}{\sqrt{T}} 
        \Big\rVert,
    \end{align*}
    which can be made arbitrarily small as $\eta$ is chosen close to $0$ on account of the tightness of $(T^{-\nicefrac{1}{2}} M^T_{Tu})$. A similar Cesàro argument shows that 
    \begin{equation*}
        \sup_{x \in [0,1]} 
        \lVert  \Big( \int_0^{T( x\land \eta)} K^T(Tx,Tx-s) \dif s \Big) \frac{M^T_{Tx} }{\sqrt{T}} 
        -
        \Big( \int_0^{T} K^T(Tx,Tx-s) \dif s \Big) \frac{M^T_{Tx} }{\sqrt{T}}  \rVert \xrightarrow[T \to \infty]{\Prob} 0,
    \end{equation*}
    and one retrieves
    \begin{equation*}
        \sup_{ \lvert u-v \rvert < \delta} \Big\lVert \textbf{\textsc{iii}} - 
        \Big\{ \Big(  \int_0^T K^T(Tv,s) \dif s \Big) 
         \frac{M_{Tv}}{\sqrt{T}} 
        -
        \Big( \int_0^T K^T(Tu,s)  \dif s \Big)\frac{M_{Tu}}{\sqrt{T}} \Big\}
        \Big\rVert \xrightarrow[T \to \infty]{\Prob} 0.
    \end{equation*}
    Then, $\lVert (  \int_0^T K^T(Tv,s) \dif s ) 
         T^{-\nicefrac{1}{2}} M^T_{Tv} 
        -
        ( \int_0^T K^T(Tu,s)  \dif s )T^{-\nicefrac{1}{2}} M^T_{Tu} \rVert  $ is bounded by
        \begin{align*}
            \sup_{ \lvert x - y \rvert \leq \delta}\big\lVert 
             \int_0^T K^T(Ty,s) \dif s  
             -
              \int_0^T K^T(Tx,s) \dif s  
              \big\rVert 
      &\sup_{t \in[0,T]}
        \lVert \frac{M^T_{t}}{\sqrt{T}}
        \rVert
        \\
        &+
       \Big(  \int_0^{\infty} \lVert \Bar{\Psi}(s) \rVert \dif s \Big)
       \sup_{ \lvert x - y \rvert \leq \delta}
        \Big\lVert \frac{M^T_{Tx}-M^T_{Ty}}{\sqrt{T}} 
        \Big\rVert,
        \end{align*}
        for any $T>0$ and any $(u,v) \in [0,1]^2$ such that $\lvert v - u \rvert  < \delta$. The desired result then deduces from Lemma~\ref{lemma:required_for_tightness} and the tightness of 
 $(T^{-\nicefrac{1}{2}} M^T_{Tu})$.
    
\end{proof}

\subsection{Further remarks on Theorem~\ref{theorem:FTCL}}\label{section:further_remarks}\label{section:multivariate_g}
Following remark~\ref{remark:multivariate_g}, replacing $k^T_{ij}(t,s)=g(\frac{t}{T}) \varphi_{ij}(t-s)$ with the coordinate-dependent kernel $k^T_{ij}(t,s)=g_{ij}(\frac{t}{T})  \varphi_{ij}(t-s)$ wherein $g \colon [0,1] \mapsto \mathcal{M}_p(\R^+)$ is only a matter of rephrasing: we have only relied on the regularity properties of $g$ and the existence of a resolvent $\sum_{i=0}^{\infty } (k^T)^{\star i}$, which is guaranteed under the condition that  $( \lVert g_{ij}\rVert_{L^{\infty}} \int_0^{\infty} \varphi_{ij}(s) \dif s )_{ij}$ has spectral radius below $1$. This is with the exception of Lemma~\ref{Lemma:IPP_base}, where the recursion step and thus the constant in our bound should be modified into
\begin{equation*}
 (\partial_t + \partial_s )(k^T)_{ij}^{\star n} (t,s)
    \leq 
    C( g,g') \frac{n}{T} \bar{\varphi}_{ij}^{ \star n}(t-s)
\end{equation*}
where $\bar{\varphi}_{ij}(t) = \lVert g_{ij} \rVert_{L^{\infty}} \varphi_{ij}(t)$ and $C( g,g') = \max_{ij} \lVert g_{ij}' \rVert_{L^{\infty}}\lVert g_{ij} \rVert_{L^{\infty}}^{-1} \mathbb{1}_{ g_{ij} \neq 0}$. In addition to this technical detail, some increased attention must then be placed on the non-commutativity of $g$ in the $\star$-iterates of the kernel. This results in more cumbersome notation for the bounds appearing throughout our proofs, hence we have kept our main results in the context of an univariate reproduction rate for the sake of simplicity. 

\begin{ack*}
     Thomas Deschatre and Pierre Gruet acknowledge support from the \textit{FiME} Lab.
\end{ack*}

\bibliographystyle{plain}
\bibliography{bibli}

\end{document}